\documentclass[9pt]{extarticle} 
\pdfoutput=1 
\usepackage[ 
  paperwidth=5.88in,paperheight=8.60in,
  textwidth=4.9in,textheight=6.7in]{geometry} 
\usepackage{subcaption}
\expandafter\def\csname ver@subfig.sty\endcsname{}
\usepackage{tikz,geometry}
\usepackage{graphicx,float}
\usepackage{pgfplots,adjustbox}  
\usetikzlibrary{external,plotmarks}
\usepackage{xcolor}
\usepackage[english]{babel}
\usepackage[utf8]{inputenc}
\usepackage{fancyhdr} 
\usepackage{amsmath, amsthm, amsfonts, amssymb, mathdots}
\usepackage{stmaryrd,ifthen}
\usepackage{paralist}
\usepackage{framed,enumitem}
\usepackage{mathtools,nicefrac}
\usepackage{multirow,subfig}
\usepackage{hyperref,url}
\usepackage[capitalize,nameinlink]{cleveref}
\usepackage[english,ruled,lined,linesnumbered]{algorithm2e}
\SetAlFnt{\small}
\AtBeginEnvironment{algorithm}{\linespread{.9}\selectfont}
\usepackage{mathtools}
\usepackage[affil-it]{authblk}
\usepackage{chngcntr}
\counterwithin{figure}{section}
\counterwithin{table}{section}
\counterwithin{figure}{section}
\counterwithin{table}{section}
\numberwithin{equation}{section}
\newtheorem{theorem}{Theorem}
\newtheorem{definition}{Definition}
\newtheorem{lemma}{Lemma}

\newtheorem{remark}{Remark}

\renewenvironment{proof}{{\noindent \bfseries Proof:}}{\qed\newline}
\addto\extrasenglish{%
}
\newcommand{\Kryl}[3]{\mathcal{K}_{#1} (#2,#3)}
\newcommand{\BKryl}[3]{\mathcal{K}^\square_{#1} (#2,#3)}

\newcommand{\diag}{\mathrm{diag}}
\newcommand{\tridiag}{\mathrm{tridiag}}
\newcommand{\Diag}{\mathrm{Diag}}
\newcommand{\sign}{\mathrm{sign}}
\newcommand{\NN}{\mathbb{N}}

\newcommand{\RR}{\mathbb{R}}

\newcommand{\OO}{\mathrm{O}}

\newcommand{\dx}{\Delta\!{X}}
\newcommand{\dv}{\Delta\!{V}}
\newcommand{\dl}{\Delta\!\Lambda}
\newcommand{\rank}{\textnormal{rank}}
\newcommand{\nullspace}{\mathcal{N}}

\newcommand{\vect}{\mathop{\mathrm{vec}}}

\newcommand{\trace}{\mathop{\mathrm{tr}}}
\newcommand{\define}{\mathop{\coloneqq}}

\newcommand{\argmin}{\operatornamewithlimits{argmin}}

\newcommand{\multcom}[1]{}

\definecolor{tured}{RGB}{190,30,60}
\definecolor{tuyellow}{RGB}{255,200,42}
\definecolor{tuorange}{RGB}{225,109,0}
\definecolor{tudarkred}{RGB}{113,28,47}
\definecolor{tugreen}{RGB}{109,131,0}
\definecolor{tulightgreen}{RGB}{172,193,58}
\definecolor{tudarkgreen}{RGB}{0,83,74}
\definecolor{tulightblue}{RGB}{102,180,211}
\definecolor{tublue}{RGB}{0,112,155}
\definecolor{tudarkblue}{RGB}{0,63,87}
\definecolor{tulightpurple}{RGB}{138,48,127}
\definecolor{tupurple}{RGB}{81,18,70}
\definecolor{tudarkpurple}{RGB}{76,24,48}
\definecolor{tublack}{RGB}{0,0,0}

\title{An inexact Matrix-Newton method for solving NEPv}
\author{Tom Werner}

\allowdisplaybreaks
\begin{document}
\affil{Institute for Numerical Analysis, TU Braunschweig\\Universit\"atsplatz 2, 38106 Braunschweig, Germany\\\href{mailto:tom.werner@tu-braunschweig.de}{tom.werner@tu-braunschweig.de}}
\maketitle
\begin{abstract}
In this paper, an inexact Newton method for solving real-valued nonlinear eigenvalue problems with eigenvector dependency (NEPv) is introduced that is able to solve the problem on a matrix level. Our main contribution is to derive a variant of Newton's method that uses global Krylov methods such as global GMRES to solve the linear operator equation necessary to compute the Newton correction in a matrix-free way. The advantages that this second order method has over the well-established SCF algorithm are explained and visualized by a variety of numerical experiments.
\end{abstract}

\section{Introduction}
In this paper, we consider real-valued eigenvector-dependent nonlinear eigenvalue problems (NEPv) to be formulated in the following form: given a continuous, symmetric matrix-valued function $H:\RR^{n,k}\rightarrow\RR^{n,n}$ (i.e. $H(V)=H(V)^T$),
\begin{equation}
\text{find } V\in\OO_{n,k}(\RR) \text{ and }\Lambda=\Lambda^T\in\RR^{k,k} \text{ such that }H(V)V=V\Lambda, \label{eq:NEPv}
\end{equation}
where $\OO_{n,k}(\RR)$ denotes the set of $n \times k$ real matrices with orthonormal columns, i.e., the $k$-dimensional Stiefel-manifold over $\RR^n$, and $X^T$ denotes the transpose of $X\in\RR^{n,k}$. NEPvs arise in a variety of applications, most notably in computational physics, such as discretized Kohn-Sham- or Gross-Pitaevskii-equations \cite{gpref1,gpref2,KSref2,KSref3}, as well as in data analysis applications for the trace ratio or robust Rayleigh quotient optimization problems used in linear discriminant analysis \cite{rlda,rldaref1,trmref1,trmref2}. In most of those cases, the eigenvalue problem is closely related to a minimization problem with orthogonality constraints and appears as an equivalent
formulation to the original problem when analyzing its first order optimality conditions (KKT-conditions) \cite{rlda,nepv,KSref1}.\\
Throughout this paper, we will assume that $H$ satisfies the right orthogonal invariance condition
$$ H(V)\equiv H(VQ),\;\text{for any}\; Q\in\OO_k(\RR),$$
meaning that every solution $(V,\Lambda)$ of \eqref{eq:NEPv} corresponds to a whole subspace of solutions 
$$\mathcal{S}_{(V,\Lambda)}=\lbrace (W,\Theta)~\vert~\exists Q\in\OO_{k}(\RR): W=VQ~\land~\Theta=Q^T\Lambda Q\rbrace.$$
We say that $(V,\Lambda)$ is a unique solution to \eqref{eq:NEPv}, if every $(\widetilde{V},\widetilde{\Lambda})$ solving \eqref{eq:NEPv} satisfies $\mathcal{R}(V)=\mathcal{R}(\widetilde{V})$, or in other words, $\mathcal{S}_{(V,\Lambda)}=\mathcal{S}_{(\widetilde{V},\widetilde{\Lambda})}$. Here, $\mathcal{R}(V)$ denotes the range of $V$, i.e. the subspace spanned by the columns of $V$.\\
Certainly, whenever $(V, \Lambda)$ satisfies \eqref{eq:NEPv}, the eigenvalues of $\Lambda$ are eigenvalues of $H(V)$. In most practical applications, one is therefore interested in finding $(V,\Lambda)$ satisfying \eqref{eq:NEPv} such that the eigenvalues of $\Lambda$ become as large/small as possible, depending on the extremal values of the corresponding optimization problem. The solutuion that allows for the easiest interpretability is the one that leads to $\Lambda$ being diagonal, since in this case, the diagonal entries $\lambda_i$ are the desired eigenvalues of $H$ with the corresponding eigenvectors $v_i$ forming he columns of $V$.\\
One can show that \eqref{eq:NEPv} has a unique solution, if there is a uniform gap between the $k$-th and $(k+1)$-st eigenvalue of $H(V)$, for any $V\in\OO_{n,k}(\RR)$, and $H$ satisfies a particular type of Lipschitz-condition \cite[Thm. 1]{nepv}.\\
Sometimes, \eqref{eq:NEPv} is modified by an additional symmetric matrix function $G:\RR^{n,k}\rightarrow\RR^{n,n}$, leading to what we will call the generalized NEPv (GNEPv): 
\begin{equation}
\text{Find } V\in\OO_{n,k}(\RR) \text{ and }\Lambda=\Lambda^T\in\RR^{k,k} \text{ such that } 	H(V)V=G(V)V\Lambda.  \label{eq:GNEPv}
\end{equation}
For simplicity, we will focus on \eqref{eq:NEPv} for the development of our algorithm and point out the necessary modifications to be able to handle \eqref{eq:GNEPv} when needed.\\   
The most widely used method to solve NEPvs is the self-consistent-field (SCF) iteration originating from electronic structure calculations, where it is commonly used to solve Hartree-Fock or Kohn-Sham equations \cite{KSref2,KSref3,KSref4}. In this simple iterative scheme, a new iterate $V_{j+1}$ is generated by computing the desired eigenpairs of $H_j=H(V_j)$ until the iteration becomes stationary. The plain algorithm is summarized in \autoref{alg:plainscf}.\\
\begin{algorithm}[H]
	\SetAlgoVlined
\KwIn{$V_0\in\OO_{n\times k}(\RR)$, $H:\RR^{n,k}\rightarrow\RR^{n,n}$}
\KwOut{$(V^*,\Lambda^*)\in\RR^{n,k}\times\RR^{k,k}$ approximate solution to $H(V)V=V\Lambda$}
Set $j=0$\\
\While{\textsc{not converged}}{
Evaluate $H_j=H(V_j)$\\
Compute $(V_{j+1},\Lambda_{j+1})$ as the $k$ largest/smallest eigenpairs of $H_j$\\
Set $j=j+1$\\
}
Return $(V^*,\Lambda^*)=(V_j,\Lambda_j)$
\vspace{-.2cm} 
	\caption{Plain SCF}
	\label[algorithm]{alg:plainscf}
\end{algorithm}
Another approach to solve \eqref{eq:NEPv}, first considered by Gao, Yang and Meza \cite{newtonvec}, is the use of Newton's method. They rewrote the eigenvalue problem \eqref{eq:NEPv} as a root finding problem $F(X)=0$ with
\begin{equation}
	F:\RR^{n,k}\times\RR^{k,k}\rightarrow\RR^{n+k,k},\text{ } F(V,\Lambda)=\begin{bmatrix} H(V)V-V\Lambda\\ V^T V -I_k
	\end{bmatrix} \label{eq:NEPvroot1}
\end{equation}
which was considered in its vectorized form $f(x) = 0$ as
\begin{equation*}
	f:\RR^{n\cdot k}\times \RR^{k\cdot k}\rightarrow\RR^{(n+k)\cdot k},\text{ } f(\vect(V),\vect(\Lambda))=\begin{bmatrix}\vect(H(V)V-V\Lambda)\\\vect(V^TV -I_k)\end{bmatrix}
\end{equation*}
and solved using a multivariate Newton approach. Here, we denote by $\vect(\cdot)$ the operator that stacks the columns of a matrix vertically.\\
The algorithm was discussed for a discretized Kohn-Sham-model for which the Jacobian $J_f(x_j)$ was explicitly available and could be shown to outperform the SCF-approach for relatively small dimensions when exploiting the structure of the problem. In this paper, we want to follow a similar route by using a Newton approach for the matrix valued problem \eqref{eq:NEPvroot}. The novel contribution of our work is to introduce a framework for efficiently finding solutions of nonlinear matrix equations, particularly \eqref{eq:NEPvroot}, by an extension of Jacobian-free Newton Krylov methods to the matrix setting. In this setting, we suggest to use global Krylov methods, in our case global GMRES, to obtain the update direction in the Newton correction equation without the need of vectorization or the explicit computation of the Jacobian.\\  
The following chapters are organized as follows: in Chapter 2, we derive the general framework for solving matrix valued root finding problems by matrix-free Newton methods. For this sake, we introduce an inexact Newton method using a global GMRES approach for approximately solving the Newton update equation. In Chapter 3, we demonstrate how to apply the inexact Newton algorithm to \eqref{eq:NEPv} or \eqref{eq:GNEPv} efficiently. We also provide a convergence analysis for our algorithm on the NEPv problem. In this, we can prove locally quadratic convergence of our method under suitable assumptions (\autoref{thm:convergence}). In Chapter 4, we apply the newly derived algorithm to several test problems from different applications and compare its convergence speed to the SCF algorithm. Additionally, we briefly comment on using global LSQR as inner solver and compare its convergence to global GMRES. We finish off with a discussion of our results and a summary of our contribution in Chapter 5.    
\section{Newton's Method for nonlinear matrix equations}
In this section, we consider the problem of finding a root $X^*\in\RR^{n,k}$ of a general, smooth matrix function $F:\RR^{n,k}\rightarrow\RR^{n,k}$. In the simple case where $k=1$, this now vector-valued problem can be solved by the well-known multivariate Newton method \cite{atkinson}, which requires the update $$x_{j+1}=x_j - J_F(x_j)^{-1} F(x_j),$$ where $J_F(x_j)\in\RR^{n,n}$ is the Jacobian of $F$ at $x_j\in\RR^n$. Usually, the explicit inversion of the Jacobian is avoided by rewriting the update formula in the form
\begin{equation}
  J_F(x_j)\Delta{x}_j=-F(x_j),\quad x_{j+1} = x_j+\Delta{x}_j, \label{eq:newtonvec}
\end{equation}
where a linear system with the Jacobian is solved to obtain the update direction $\Delta{x}_j$. A similar approach to \eqref{eq:newtonvec} is possible for matrix functions ($k>1$) using the Fréchet derivative of $F$ instead of the Jacobian to perform the update
\begin{equation}
  L_F(X_j,\dx_j)=-F(X_j),\quad X_{j+1}=X_j+\dx_j, \label{eq:newtonfrechet}
\end{equation}
where $L_F(X_j,\dx_j)$ is the Fréchet derivative of $F$ at $X_j$ in direction of $\dx_j$.
\begin{definition}[Fréchet derivative and Kronecker form] $ $\\
The \textbf{Fréchet derivative} of a matrix function $F:\RR^{n,k}\rightarrow\RR^{n,k}$ at a point $X$, if it exists, is the unique linear mapping
$$L_F(X):\RR^{n,k}\rightarrow\RR^{n,k},\quad \dx\mapsto L_F(X)(\dx)\define L_F(X,\dx),$$
such that for every $\dx\in\RR^{n,k}$ it holds
$$ F(X+\dx)-F(X)-L_F(X,\dx)=\mathop{o}(\lVert \dx \rVert).$$
Since $L_F(X)$ is a linear operator, there exists $K_F(X)\in\RR^{n\cdot k,n\cdot k}$ not depending on $\dx$ such that
$$ \vect(L_F(X,\dx))=K_F(X) \vect(\dx).$$
We refer to $K_F(X)$ as the \textbf{Kronecker form} of the Fréchet derivative. We say that $L_F(X)$ is \textbf{nonsingular} at a point $X$, if its matrix representation $K_F(X)$ is nonsingular. 
\end{definition}
It is straightforward to show that the classic rules for differentiating composed functions apply to the Fréchet derivative in the same way, most notably the product and chain rule \cite[Sec. 3.2]{matrixfunctions}.\\
As the Fréchet derivative is linear in $\dx$, we can see immediately that solving a linear matrix equation is required to find the update direction $\dx_j$ in \eqref{eq:newtonfrechet}. Certainly, similar to the scalar and vector-valued case, Newton's method in the form \eqref{eq:newtonfrechet} can be shown to converge quadratically in a neighbourhood around a root $X^*\in\RR^{n,k}$ of $F$ as long as $F$ is Fréchet differentiable and the update equation $L_F(X_j,\dx_j)=-F(X_j)$ has a unique solution for every $X_j$ \cite{newtonbanach}. However, there is also some research towards the convergence of Newton's method if the Jacobian (or the Fr\'echet derivative) is analytically or numerically rank deficient and the update equation has only unique solutions in the least-squares sense or non-isolated zeros exist \cite{newtonpseudoinverse,newtonconstrank,newtoncriterion,nonisolated}.\\
Since solving the linear equation \eqref{eq:newtonfrechet} directly for $\dx_j$ is usually challenging due to the Fréchet derivative being complicated, we have to settle for approximate solutions to the update equation which we can compute by iterative methods, leading to a class of algorithms called inexact Newton methods (INM) which we will briefly introduce in the following section.
\subsection{Inexact Newton methods}\label{sec:imn}
In the inexact Newton framework, we seek to find an update direction $\dx_j$ by finding a matrix satisfying the residual inequality
\begin{equation}
\lVert R_j(\dx_j)\rVert_F \leq \eta_j \lVert F(X_j)\rVert_F,\quad \eta_j>0, \label{eq:inexactupdate}
\end{equation}
where $R_j(\dx_j)\define L_F(X_j,\dx_j) +F(X_j)$ is the residual of the correction equation \eqref{eq:newtonfrechet} and $\eta_j$ is the so-called \emph{forcing term} affecting the relative accuracy at which the equation is solved \cite{forcingterm3,forcingterm}. Since $L_F(X)$ is linear, a solution $\dx_j$ of \eqref{eq:inexactupdate} can be computed by employing iterative solvers for linear equations, in particular Krylov subspace methods. Note that in the Krylov subspace setting, we usually do not need to know the exact derivative of $F$ to be able to extend the subspace, we only need to be able to apply the derivative in a given direction, which can be done by finite difference approximations. In the case of vector-valued functions, approaches combining inexact Newton methods with Krylov subspace methods as (matrix-free) inner solver are referred to as \emph{(Jacobian-free) Newton-Krylov methods} (JFNK) \cite{jfnk}.\\
To get more insight on how to apply JFNK to matrix-valued problems, consider again the Newton update equation \eqref{eq:newtonfrechet} and its vectorized companion
\begin{equation}
K_F(X_j) \vect(\dx_j) = - \vect(F(X_j)) \label{eq:newtonkron}
\end{equation}
using the Kronecker form of the Fréchet derivative. In theory, $\dx_j$ can be computed by employing GMRES for the vectorized equation \eqref{eq:newtonkron}, where we have to compute the matrix-vector products $K_F(X_j)\cdot v_\ell$ to expand the Arnoldi sequence. However, computing $K_F(X_j)$ explicitly is not appropriate in this application, since the computation itself is expensive and the resulting problem is way bigger than the original matrix-valued problem.\\
Instead, we want to view \eqref{eq:newtonfrechet} as an equation with multiple right-hand sides  
$$\mathcal{A}(Y)=B, \quad Y,B\in\RR^{n,k},\;\mathcal{A}:\RR^{n,k}\rightarrow\RR^{n,k}$$  
for the linear operator $\mathcal{A}$ and follow a global GMRES approach \cite{globalgmres}. In the global GMRES-algorithm, we seek to find an approximate solution to $\mathcal{A}(Y)=B$ by spanning the block Krylov subspace
$$\BKryl{\ell}{\mathcal{A}}{B}=\textnormal{\textbf{blockspan}}\lbrace B,\mathcal{A}(B),\mathcal{A}^2(B),\dots,\mathcal{A}^{\ell-1}(B)\rbrace,$$
where $\mathcal{A}^s(B)=\mathcal{A}(\mathcal{A}^{s-1}(B))$ and $\mathcal{A}^0(B)=B$, with a series of matrices $\lbrace V_1,V_2,\dots,V_\ell\rbrace$ generated from a global Arnoldi-process that are orthogonal to each other with respect to the Frobenius inner product $\langle X,Y\rangle _F= \trace(X^TY)$ and normalized, i.e. $\langle V_i,V_j\rangle_F=\delta_{i,j}$, and satisfy the relation  
$$\mathcal{A}\mathcal{V}_\ell = \mathcal{V}_{\ell+1}(\underline{H}_\ell \otimes I_k),$$
where $\mathcal{V}_\ell=\left[V_1,V_2,\dots,V_\ell\right]\in\RR^{n,k\cdot\ell}$ and $\underline{H}_\ell\in\RR^{\ell+1,\ell}$ is upper Hessenberg with an additional row and entries $h_{m,i}=\langle V_m,V_i\rangle_F$, $i=1,\dots,\ell$, $i\leq m+1$. Given some initial $X_0\in\RR^{n,k}$, global GMRES attempts to find the solution $X_\ell$ that minimizes the residual
$$ \lVert R_\ell\rVert_F:=\lVert B-\mathcal{A}(X_\ell)\rVert_F=\min_{Z\in\BKryl{\ell}{\mathcal{A}}{R_0}}\lVert B - \mathcal{A}(X_0+Z))\rVert_F=\min_{c\in\RR^\ell}\left\lVert \underline{H}_\ell c -  \lVert R_0\rVert_F e_1 \right\rVert_2$$ 
over the Krylov subspace $\BKryl{\ell}{\mathcal{A}}{R_0}$. The following Lemma is a direct consequence of the relation between GMRES and global GMRES and justifies using a global GMRES approach for solving the update equation \eqref{eq:newtonfrechet}:
\begin{lemma}\label{lem:gmres}
Let $F:\RR^{n,k}\rightarrow\RR^{n,k}$ be Fréchet differentiable at $X\in\RR^{n,k}$ with Fréchet derivative $L_F(X):\RR^{n,k}\rightarrow\RR^{n,k}$ and Kronecker form $K_F(X)\in\RR^{n\cdot k,n\cdot k}$. Then, for any $B\in\RR^{n,k}$, if neither iteration breaks down, the approximate solution $Y_\ell\in\RR^{n,k}$ to the matrix equation
$$L_F(X,Y)=B$$
after $\ell$ steps of the global GMRES method and the approximate solution $y_\ell\in\RR^{n\cdot k}$ to the linear system
$$K_F(X)y=\vect(B)$$
after $\ell$ steps of the GMRES method satisfy $y_\ell=\vect(Y_\ell)$, given that $y_0=\vect(Y_0)$.
\end{lemma}
\begin{proof}
Let $Y_0\in\RR^{n,k}$ be the initial guess for the global GMRES algorithm (GL-GMRES) and $y_0=\vect(Y_0)\in\RR^{n\cdot k}$ be the corresponding initial guess for the GMRES algorithm. Using the definition of the Kronecker form and the fact that $\langle V_i,V_j\rangle_F=\langle \vect(V_i),\vect(V_j)\rangle_2$, for matrices $V_i,V_j\in\RR^{n,k}$, one can easily verify by induction that the basis matrices $V_i$ of the Block-Krylov subspace $\BKryl{\ell}{L_F(X)}{R_0}$ and the basis vectors $v_i$ of the standard Krylov subspace $\Kryl{\ell}{K_F(X)}{r_0}$ satisfy $v_i=\vect(V_i)$, $i=1,\dots,\ell$, where $R_0=B-L_F(X,Y_0)$ and $r_0=\vect(B)-K_F(X)y_0$. Therefore, we also have the equivalence of the Hessenberg matrices $\underline{H}^{\scalebox{.5}{GMRES}}_\ell$ and $\underline{H}^{\scalebox{.5}{GL-GMRES}}_\ell$ and the two minimization problems
$$ \min_{c\in\RR^\ell}\left\lVert \underline{H}^{\scalebox{.5}{GL-GMRES}}_\ell c - \lVert R_0 \rVert_F e_1\right\rVert_2  \quad\text{and}\quad \min_{c\in\RR^\ell}\left\lVert \underline{H}^{\scalebox{.5}{GMRES}}_\ell c - \lVert r_0 \rVert_2 e_1\right\rVert_2$$
have the same solution $c^*=(c_1,\dots,c_\ell)^T\in\RR^\ell$. It immediately follows that 
$$y_\ell = y_0 + \sum_{i=1}^\ell c_i v_i = \vect(Y_0) + \sum_{i=1}^\ell c_i \vect(V_i)=\vect(Y_0 + \sum_{i=1}^\ell c_iV_i)=\vect(Y_\ell).$$ 
and both algorithms are mathematically equivalent.
\end{proof}
From \cref{lem:gmres}, we know that solving \eqref{eq:newtonfrechet} by global GMRES is equivalent to solving \eqref{eq:newtonkron} by GMRES. However, the global GMRES algorithm avoids the computation of the Kronecker form $K_F(X_j)$ and is able to operate entirely on the original matrix space $\RR^{n,k}$ using level three BLAS-routines for the application of $L_F(X_j)$ and the Gram-Schmidt-process, thus avoiding large scale matrix-vector multiplications and potentially allowing the usage of structure inherent to the problem, such as low-rank or sparsity structure. The global GMRES algorithm for \eqref{eq:newtonfrechet} is summarized in \cref{alg:glgmres}.\\
Note that the minimization problem with the upper Hessenberg matrix $\underline{H}_\ell\in\RR^{\ell+1,\ell}$ in step 9 can be solved by finding $c$ satisfying the linear system $Rc=\xi$, where $\underline{H}_\ell=QR$ is a QR-decomposition of $\underline{H}_\ell$ and $\xi=Q^T \lVert R_0\rVert_F e_1$. In most available software packages, this is done by transforming $\underline{H}_\ell$ to upper triangular form using Givens rotations for every newly computed column of $\underline{H}_\ell$. This approach can be used in \cref{alg:glgmres} as well to avoid unnecessarily large Krylov spaces. Certainly, the algorithm can be easily modified to apply preconditioning to the system.\\
Additionaly, we want to mention again that it is not necessary to explicitly know the Fréchet derivative of $F$ at $X$ to be able to expand the Krylov subspace. If $L_F(X)$ can not be computed explicitly, it usually suffices to approximate $L_F(X,V_i)$ by, e.g. finite differences or, if accurate approximations are needed, the complex step\cite{complexstep}.\\
\begin{algorithm}[H]
		\SetAlgoVlined
	\KwIn{$X\in\RR^{n,k}$, $\dx_0\in\RR^{n,k}$ initial guess, $\eta<1$, $\ell\in\mathbb{N}$}
	\KwOut{$\dx\in\RR^{n,k}$ approximate Newton update direction}
	Compute $B=-F(X)$ and $R_0=B-L_F(X,\dx_0)$\\
	Define $V_1=\frac{R_0}{\lVert R_0\rVert_F}$ \\
	\For{$i=1,2,\dots,\ell$}{
		Compute $W=L_F(X,V_i)$ \\
		\For{$m=1,\dots,i$}{
			$ h_{m,i} = \langle V_m,W\rangle _F$, \quad $W = W - h_{m,i} V_m$
		}
		Expand $h_{i+1,i} = \lVert W\rVert _F$, \quad $V_{i+1}=\frac{W}{h_{i+1,i}}$  
	}
	Solve the least-squares problem $\min_{c\in\RR^\ell} \left\lVert \underline{H}_\ell c - \lVert R_0\rVert_F e_1\right\rVert _2$ to obtain $c=(c_1,\dots,c_\ell)^T$ \label{step:lssolve}\\
	Construct the approximate solution $\dx=\dx_0+\sum_{i=1}^\ell c_i V_i$\\
	\eIf{$\lVert B-L_F(X,\dx)\rVert_F\leq \eta \lVert B\rVert _F$}{
		\Return $\dx$
	}{
	Set $\dx_0=\dx$ and restart the algorithm
	}
\vspace{-.2cm} 
	\caption{Global GMRES for \eqref{eq:inexactupdate}}
	\label[algorithm]{alg:glgmres}
\end{algorithm}
In the following chapter, we will give a detailed description on how to use this approach for solving NEPvs efficiently without giving up the original matrix equation structure. 
\section{The inexact Matrix-Newton method for NEPv}
Consider problem \eqref{eq:NEPv} once again in the original matrix root finding form \eqref{eq:NEPvroot1}. For symmetry reasons, we change the sign of the orthogonality constraint to obtain the nonlinear function
\begin{equation}
	F:\RR^{n+k,k}\rightarrow\RR^{n+k,k},\text{ } F(X)\define F\left(\begin{bmatrix} V\\ \Lambda \end{bmatrix}\right)=\begin{bmatrix} H(V)V-V\Lambda\\ I_k-V^TV
	\end{bmatrix}.	 \label{eq:NEPvroot}
\end{equation}
To derive an efficient algorithm for solving \eqref{eq:NEPv} in its original form, we want to adapt the techniques introduced in \cref{sec:imn} to the particular matrix function from \eqref{eq:NEPvroot}. For simplicity, we will focus on \eqref{eq:NEPv} here but it is straightforward to extend the ideas to \eqref{eq:GNEPv}.\\
For Newton's method to be applicable, we require $F$ to be Fréchet differentiable on $\RR^{n+k,k}$. Let $\dx=\begin{bmatrix} \dv\\ \dl \end{bmatrix}\in\RR^{n+k,k}$ be an arbitrary direction matrix and let $F_1(X):=H(V)V-V\Lambda$ and $F_2(V):=I_k-V^TV$ be the two parts of $F(X)$. Then, assuming that $H$ is Fr\'echet differentiable at $V$, we have by the product rule
\begin{equation}
	L_{F_1}(X,\dx) = H(V) \dv + L_H(V,\dv)V - (V \dl + \dv \Lambda). \label{eq:frechetnepv}
\end{equation}
In consequence, $F_1(X)$ is Fréchet differentiable whenever $H(V)$ is Fréchet differentiable. One can immediately extend this idea to the GNEPv \eqref{eq:GNEPv} if both $G(V)$ and $H(V)$ are Fréchet differentiable. For $F_2(V)$, it is easy to see that 
\begin{equation}
  L_{F_2}(V,\dv) = -(V^T \dv + \dv^T V) \label{eq:frechetorthogonal}
\end{equation}
and $F_2$ is always Fréchet differentiable. Thus, smoothness of $F(X)$ directly depends on the smoothness of $H(V)$ and $G(V)$.
\begin{remark}
We want to note that for complex problems, the orthogonality constraint $I-V^HV$ is not Fr\'echet differentiable anymore. Intuitively, the approach can still be applied in a similar way since Newton's method to $F_2(V)$ alone leads to the well-known Newton-Polar-iteration \cite{polar} which converges quadratically, even if $V$ is complex. Numerical experiments also provide strong evidence that the approach can be applied to complex problems with minor modifications, however, providing the theoretical background for that is still a point of research. 
\end{remark}
\subsection{Pre and Postprocessing by SCF}
As mentioned earlier, the most commonly used method when solving \eqref{eq:NEPv} is the SCF algorithm, originating from Hartree-Fock and Kohn-Sham density functional theory. The plain algorithm (\cref{alg:plainscf}) is easy to implement using an arbitrary algorithm for computing a limited number of eigenpairs of the matrix $H(V_j)$, which is available in almost every current programming language. Additionally, the algorithm can be shown to be globally convergent if there exists a uniform gap between the $k$-th and $(k+1)$-st eigenvalue of every $H(V_j)$, where $\lbrace V_j\rbrace_j$ is the sequence of SCF iterates, and $H$ satisfies the Lipschitz-like condition 
$$ \lVert H(V)- H(\widetilde{V})\rVert_2\leq \xi \lVert \sin(\Theta(V,\widetilde{V}))\rVert_2,\quad \forall~ V,\widetilde{V}\in\OO_{n,k}(\RR),$$
where $\Theta(V,\widetilde{V})$ is a diagonal matrix containing the canonical angles between $V$ and $\widetilde{V}$ \cite[Thm. 3]{nepv}. However, the main drawback of the SCF algorithm is that this convergence is generally of linear order only.\\
Note that several techniques for improving convergence speed and stability of the SCF iteration have been developed. These techniques include level-shifting methods to increase the gap between the $k$-th and $(k+1)$-st eigenvalue \cite{optimalscf,levelshift}, damping and trust-region strategies \cite{dampscf,KSref4,KSref1} and sequential subspace searching \cite{convexmin}. However, those are not considered in this paper as they usually require problem specific adaptations to make them worth while and we want to keep our methods as simple and universally applicable as possible.\\
For Newton's method, Kantorovich's theorem guarantees locally quadratic order of convergence of the sequence $\lbrace X_j\rbrace_j$, given that the update equation can be solved exactly \cite{kantorovich}. For inexact Newton methods, we can usually expect superlinear convergence close to the desired solution \cite{inexactnewton}.  \\
In our approach, we combine the benefits of both algorithms: we begin by applying the SCF iteration to an arbitrary initial guess (with orthonormal columns) and switch to Newton's method once we are "close enough" to a solution. This way, we can let SCF find the desired subspace and exploit the faster order of convergence of Newton's method once we are on the trajectory of convergence. After we found a solution by Newton's method, we suggest to perform an additional SCF step with this solution. This postprocessing step guarantees the matrix $\Lambda$ to be diagonal as opposed to just being symmetric to allow for better interpretability and comparability to the SCF solution. Additionally, if the eigenvalues before and after postprocessing are equal, the SCF algorithm accepts the solution as a global optimizer.
\subsection{Local convergence analysis}
One drawback of Newton's method for \eqref{eq:NEPvroot} which was already pointed out in \cite{newtonvec} is that, due to the constraint $V^TV=I_k$ and the orthogonal invariance, the derivative will always be singular at the solution. To be more precise, the derivative will be of rank $(n+k)k - \frac{k(k-1)}{2}=nk + \frac{k(k+1)}{2}$. However, the rank of the derivative matches the total amount of unknowns in the equation since the strict triangle of $\Lambda$ is directly determined by its symmetry, leading to $\sum_{i=1}^{k-1}i=\frac{k(k-1)}{2}$ redundant variables.\\

It has been pointed out first by Ben-Israel in 1966 \cite{newtonpseudoinverse} that Newton's method for functions with singular Jacobians can still be carried out by replacing the inverse of the Jacobian by its pseudoinverse to get the iteration
\begin{equation}
	x_{j+1}= x_j - J_f(x_j)^+f(x_j). \label{eq:newtonpinv}
\end{equation}
However, the convergence theory in this work requires the limit point to be an isolated zero of $f$. Lately, there has been extensive research to formulate more practical convergence criteria for Newton's method in the form \eqref{eq:newtonpinv} under the additional assumption that $J_f$ has locally constant rank \cite{newtonconstrank,newtoncriterion} and proof local quadratic convergence of this iteration, even for non-isolated zeros \cite{nonisolated}. To characterize the convergence of \eqref{eq:newtonpinv}, we need the definition of a semi-regular zero:
\begin{definition}[Dimension of a zero and semi-regularity (\cite{nonisolated})]
Let $x^*$ be a zero of a smooth mapping $f:\Omega\subset\RR^m\rightarrow\RR^n$, i.e., $f(x^*)=0$.
\begin{enumerate}
\item If there is an open neighborhood $\mathcal{B}\subset\Omega$ of $x^*$ in $\RR^m$ such that $\mathcal{B}\cap f^{-1}(0) = \phi(\Gamma)$, where $z\mapsto\phi(z)$ is a differentiable injective mapping defined in a connected open set $\Gamma\subset\RR^p$, for a certain $p>0$, which satisfies $\phi(z^*)=x^*$ and $\rank(J_\phi(z^*))=p$, then the \textbf{dimension of $x^*$ as a zero of $f$} is defined as 
$$\dim_f(x^*) :=\rank(J_\phi(z^*)) = p$$ 
As a special case, an isolated zero has dimension zero by definition.
\item Let $\Gamma^*\subset\Gamma$ be an open neighborhood around $z^*$ such that $\rank(J_\phi(\hat{z}))=p$, for every $\hat{z}\in\Gamma^*$. We then say that every $\hat{x}\in\phi(\Gamma^*)$ is in the same \textbf{branch of zeros} as $x^*$.
\item We say that $x^*$ is a \textbf{semi-regular zero} of $f$, if its dimension is well-defined and identical to the nullity of the Jacobian of $f$ at $x^*$, i.e. 
$$\dim_f(x^*) \define \rank\left(\nullspace(J_f(x^*))\right).$$
\item We say that a matrix $X^*\in\RR^{n,k}$ is a semi-regular zero of a smooth matrix-valued function $F:\Omega\subset\RR^{n,k}\rightarrow\RR^{n,k}$, if its vectorization $x^*=\vect(X^*)\in\RR^{n\cdot k}$ is a semi-regular zero of the vectorized function $f:\vec{\Omega}\subset\RR^{n\cdot k}\rightarrow\RR^{n\cdot k}$, $\vect(X)\mapsto \vect(F(X))$.
\end{enumerate}
\end{definition} 
To proof convergence of \eqref{eq:newtonpinv} for NEPv, let us first consider its vectorized companion \eqref{eq:newtonkron} again and denote $x_j\define \vect(X_j)$ and $f(x)\define\vect(F(X))$. To obtain a new update direction $\Delta{x}_j$, we need to be able to compute the least-squares solution 
\begin{equation}
\Delta{x}_j=-K_F(X_j)^+f(x_j)=\argmin_s \lVert f(x_j)+ K_F(X_j)s\rVert_2 \label{eq:lssolve}
\end{equation}
in every Newton step. By \cite[Thm 2.6]{singulargmres} and \cite[Thm. 2.8]{geometricgmres}, we have the following:
\begin{lemma}\label{lem:gmresls}
Let $\mathcal{A}:\RR^{n,k}\rightarrow\RR^{n,k}$ be a linear operator, $B\in\RR^{n,k}$, and consider the linear operator equation
\begin{equation}
 \mathcal{A}(X)=B. \label{eq:linop}
\end{equation} 
Then, if \eqref{eq:linop} is consistent, i.e., if $B\in\mathcal{R}(\mathcal{A})$, GL-GMRES finds a least-squares solution $X^*$ to \eqref{eq:linop} (w.r.t. the Frobenius norm) for any initial value $X_0\in\RR^{n,k}$ without breaking down if and only if $\mathcal{R}(\mathcal{A})\cap\nullspace(\mathcal{A})=\lbrace 0\rbrace$. 
\end{lemma} 
\begin{proof}
The result is a direct consequence of the connection between GL-GMRES for \eqref{eq:linop} and GMRES for its vectorization. For GMRES, the result was already proven in \cite{geometricgmres}.
\end{proof}  
We want to note that there is a variety of work suggesting techniques for solving (nearly) singular linear systems by Krylov subspace methods, including the use of singular preconditioners in \cite{gmressingularprecon} or the development of breakdown-free variants of GMRES (BFGMRES) in \cite{bfgmres}. These techniques can also be adapted to work in GL-GMRES in a similar way, if the plain algorithm does not yield satisfactory results. Additionally, we will comment on using iterative solvers designed for least-squares problems, such as LSQR, in \autoref{sec:lsqr}. \\
The following Lemma states that locally, the Newton correction equation is consistent:
\begin{lemma}[Consistency of the update equation] \label{lem:consistency}
Let $\mathbb{S}_k\define \lbrace S\in\RR^{k,k}:S=S^T\rbrace$ denote the set of symmetric $k\times k$-matrices and consider the NEPv update equation
\begin{equation}
L_F(X,\dx) \define \begin{bmatrix}H(V)\dv + L_H(V,\dv)V - (V\dl + \dv \Lambda)\\ -(V^T\dv + \dv^TV)\end{bmatrix} = \begin{bmatrix}H(V)V-V\Lambda\\I_k-V^TV\end{bmatrix}=F(X). \label{eq:nepvcorrection}
\end{equation}
\begin{enumerate}
\item[(i)] For every $X=(V,\Lambda)\in\RR^{n,k}\times\RR^{k,k}$, we have 
$$\mathcal{R}(L_F(X)) \subseteq (\RR^{n,k}\oplus \mathbb{S}_k) \subsetneq (\RR^{n,k}\oplus \RR^{k,k}),$$
i.e. $L_F(X)$ is always rank-deficient.
\item[(ii)] Let $X^*=(V^*,\Lambda^*)$ be a solution of \eqref{eq:NEPv} for which the operator 
\begin{equation}
L_{\widetilde{F}}(X^*,\dx) \define \begin{bmatrix}H(V^*)\dv + L_H(V^*,\dv)V^* - (V^*\dl + \dv \Lambda^*)\\ -{V^*}^T\dv \end{bmatrix} \label{eq:nepvcorrectionmod}
\end{equation}
is locally invertible. Then, for all $X=(V,\Lambda)\in\OO_{n,k}(\RR)\times\mathbb{S}_k$ sufficiently close to $X^*$, we have 
\begin{equation}
\mathcal{R}(L_F(X))\equiv (\RR^{n,k}\oplus\mathbb{S}_k),
\end{equation}
i.e. the correction equation \eqref{eq:nepvcorrection} is locally consistent.
\end{enumerate}
\end{lemma}
The interested reader is referred to the appendix for a proof of this result. By \autoref{lem:consistency}, we have that GL-GMRES will find a least squares solution to $$L_F(X_j,\dx_j)=-F(X_j)$$ in every step, if $\mathcal{R}(L_F(X_j))\cap\nullspace(L_F(X_j))=\lbrace 0\rbrace$, for every $j$. We denote the solution by $\dx_j=-L(X_j)^+F(X_j)$. We can now provide the following convergence result:
\begin{theorem}[Convergence of Newton's method for NEPv] \label{thm:convergence}
Let $H:\RR^{n,k}\rightarrow\RR^{n,n}$ be a smooth, symmetric matrix function and consider the NEPv-root-finding function 
$$F(X) = \begin{bmatrix}H(V)V-V\Lambda\\I_k-V^TV\end{bmatrix}.$$
Let $\Omega\subset\RR^{n+k,k}$ be an open domain containing a semi-regular zero $X^*$ of $F$ in which $\rank(L_F(X))=nk+\frac{k(k+1)}{2}$, for every $X\in\Omega$. Further, assume that the the conditions of \autoref{lem:gmresls} and \autoref{lem:consistency} are locally satisfied, i.e., one can always compute a least-squares solution to the update equation using GL-GMRES.\\
Then, for every open neighborhood $\Omega_1$ of $X^*$, there is a neighborhood $\Omega_0$ of $X^*$, such that, from every initial iterate $X_0\in\Omega_0$, the Newton iteration 
$$X_{j+1} = X_j - L_F(X_j)^+F(X_j)$$
converges quadratically to a semi-regular zero $\hat{X}\in\Omega_1$ of $F$ in the same branch as $X^*$.
\end{theorem} 
\begin{proof}
By \eqref{eq:frechetnepv} and \eqref{eq:frechetorthogonal}, $F$ is smooth if $H$ is smooth. Also, by \autoref{lem:gmresls} and \autoref{lem:consistency}, solving the update equation by GL-GMRES is always possible and gives a least-squares solution to 
$$\min_{\dx\in\RR^{n+k,k}} \lVert L_F(X_j,\dx)+F(X_j)\rVert_F=\min_{\Delta{x}\in\RR^{(n+k)\cdot k}}\lVert K_F(X_j)\Delta{x}+\vect(F(X_j))\rVert_2.$$ 
Thus, the formula 
$$X_{j+1}=X_j + \dx_j,\quad L_F(X_j,\dx_j)=-F(X_j),$$
is equivalent to the vectorized update 
$$x_{j+1}=x_j - J_f(x_j)^+f(x_j)$$
according to \eqref{eq:lssolve}. The quadratic convergence of Newton's method then follows by \cite[Thm. 4.1]{nonisolated}.
\end{proof}
We want to give a few quick notes on \autoref{thm:convergence}: First of all, the conditions of \autoref{lem:gmresls} might be hard to verify for a specific problem a priori without any knowledge of $H$. Also, \autoref{lem:gmresls} guarantees convergence for every initial guess in GL-GMRES, which is a bit more than we need since usually, no better guess than the all zero matrix is available for an arbitrary problem. Additionally, one might be able to weaken the assumption by using a different inner solver than plain GL-GMRES, as was pointed out earlier.\\
Secondly, in the inexact Newton framework presented before, the quadratic convergence from \autoref{thm:convergence} will generally not be achievable due to the inexactness of the solve. However, the order of convergence should still be faster than linear in practice, giving the algorithm an advantage over SCF. Since the solve of the update equation gets more accurate the closer we get to the solution, we can still locally achieve almost quadratic convergence. Additionally, the low accuracy in the first steps of Newton's method requires a small Krylov subspace in GL-GMRES and reduces the chance for break-downs in the ill-conditioned solve as the Krylov subspace will be of size way smaller than $nk+\frac{k(k+1)}{2}$.\\
Thirdly, the existence of a neighborhood around a semi-regular zero usually requires us to be sufficiently close to the solution when starting the iteration process, which is a well-known blemish of Newton's method. However, our idea of preprocessing by SCF should bring us \glqq close enough\grqq~to a semi-regular zero, where the question of what actually is \glqq close enough\grqq~is still a point of current research. In the numerical experiments in \autoref{sec:numerical}, there have been empirically good choices for every experiment determined by the current value of the Newton residual $\lVert F(X_j)\rVert_F$ and the effect of switching too early or too late is visualized. 
Lastly, a result similar to \autoref{thm:convergence} should be possible for GNEPv as well, if we assume that the additional function $G(V)$ and $L_G(V)$ still leave us with a locally consistent update equation according to \autoref{lem:consistency}. This is not done in this paper but experimentally showed to be valid.   
\subsection{Implementation details}
In the following, we will provide information on how to implement the steps from \autoref{alg:imn} efficiently, particularly the GL-GMRES for solving the update equation and the parameter choice in the inexact Newton process, that is choosing the forcing terms and damping factors.
\subsubsection{GL-GMRES for NEPv}
As derived earlier, the Fréchet derivative $L_F(X,\dx)$ in the update equation has the form
\begin{equation}
	L_F(X,\dx) = \begin{bmatrix} H(V)\dv + L_H(V,\dv)V - (V \dl + \dv \Lambda) \\ -(V^T\dv + \dv^TV)\end{bmatrix} \label{eq:frechetupdatenepv}
\end{equation}
for the NEPv \eqref{eq:NEPv} and  
\begin{equation}
	L_F(X,\dx) = \begin{bmatrix} H(V)\dv + L_H(V,\dv)V - \left(L_G(V,\dv)V\Lambda + G(V)(V \dl + \dv \Lambda)\right) \\ -(V^T\dv + \dv^TV)\end{bmatrix} \label{eq:frechetupdategnepv}
\end{equation}
for the GNEPv \eqref{eq:GNEPv}. In both cases, we are able to recycle the matrices $H(V)$ and $G(V)$ since they do not depend on $\dx$, meaning that the main effort in the evaluation of $L_F(X,\dx)$ is in the evaluation of $L_G(V,\dv)$ and $L_H(V,\dv)$. This is particularly appealing when no closed form for the Fréchet derivatives is available since $H(V)$ and $G(V)$ are available and can be reused in finite difference approximations 
$$ L_H(V,\dv)\approx \frac{H(V+h\dv)-H(V)}{h}$$
as. Thus, the effort in every Arnoldi step essentially reduces to one function evaluation of $H$.\\
We want to note than in some applications, it might be more suitable to compute the action $L_H(V,\dv)\cdot V$ of the Fréchet derivative on the matrix $V$ rather than forming the square matrix explicitly, thus avoiding the large matrix-matrix-multiplication. Special cases where this approach is beneficial include the Fréchet derivative having a low-rank structure, i.e. $L_H(V,\dv)=WY^T$, $W,Y\in\RR^{n,p}$, $p\ll n$.\\
Apart from the evaluation of $L_F(X,\dx)$ in every step, the second most time consuming operation is the orthogonalization of the newly computed basis vector against the existing basis vectors $V_i$. In \cref{alg:glgmres}, this orthogonalization is achieved by a modified Gram-Schmidt (MGS) procedure, which tends to get very expensive when the size of the Krylov space grows. The standard way to avoid the subspace growing too large is to employ restarts in GL-GMRES. This allows for a less time-consuming orthogonalization and requires less storage but, in practice, can lead to stagnating Newton processes close to the solution in some examples since the necessary accuracy was not achievable in the small Krylov subspace. In those cases, using deflation or weighting techniques might improve the convergence of GL-GMRES \cite{globalgmres2}.\\
In addition to that, we want to note that explicitly forming the matrix product $X^TY$, $X,Y=\RR^{n,k}$, when computing the Frobenius scalar product $\langle X,Y\rangle _F=\trace (X^TY)$ should be avoided since the computation of $X^{T}Y$ requires $\mathcal{O}(n^{2}\cdot k)$ flops and its off-diagonal entries are not needed in the trace. Instead, we suggest to use approaches requiring $\mathcal{O}(n\cdot k)$ operations, such as 
$$\langle X,Y\rangle_F = \sum_{j=1}^k \langle x_j,y_j\rangle_2 = \sum_{i=1}^n \sum_{j=1}^k x_{i,j}\cdot y_{i,j},$$
where $x_j,y_j$ are the columns of $X$ and $Y$, respectively. In MATLAB, both approaches were simililarly fast in our experiments, where we use \texttt{dot(X,Y,1)} to compute the column-wise scalar products $\langle x_j,y_j\rangle_F$ and elementwise-multiplication \texttt{X.*Y} for the last expression.
\subsubsection{Forcing terms and damping}\label{sec:eisenstatwalker}
For inexact Newton methods to be efficient, the choice of the forcing terms $\eta_j$, i.e. the relative accuracy at which the update equation is solved, and the damping technique are a crucial element. Balancing between the required Newton steps and the required GL-GMRES steps for solving the update equation is key. For our algorithm, we adapt the  parameters suggested by Eisenstat and Walker \cite{inexactnewton} in the vector-valued setting for our matrix-valued algorithm. 
\begin{enumerate}[leftmargin=\parindent,align=left,labelwidth=\parindent,labelsep=0pt]
\item[Choice 1: ] Given an initial $\eta_0\in[0,1)$, define 
\begin{equation}
	\eta_j = \frac{\bigl\lvert \lVert F(X_j)\rVert_F - \lVert R_{j-1}\rVert_F\bigr\rvert}{\lVert F(X_{j-1})\rVert_F},\quad j=1,2,\dots . \label{eq:forcingC1}
\end{equation}
To avoid oversolving, i.e. $\eta_j$ decreasing unnecessarily fast, we modify $\eta_j=\max\lbrace \eta_j, \eta_{j-1}^{(1+\sqrt{5})/2}\rbrace$, whenever $\eta_{j-1}^{(1+\sqrt{5})/2}>0.1$.
\item[Choice 2: ] Given an initial $\eta_0\in[0,1)$, $\gamma\in[0,1]$ and $\alpha\in(1,2]$, define
\begin{equation}
	\eta_j = \gamma \left(\frac{\lVert F(X_j)\rVert_F}{\lVert F(X_{j-1})\rVert_F}\right)^\alpha, \quad j=1,2,\dots .\label{eq:forcingC2}
\end{equation}
Again, we modify $\eta_j=\max\lbrace\eta_j, \gamma \eta_{j-1}^\alpha\rbrace$, whenever $\gamma \eta_{j-1}^\alpha >0.1$. According to Eisenstat and Walker, defining $\gamma=0.9$ and $\alpha=\frac{1+\sqrt{5}}{2}$ lead to empirically good results, which is why we choose these parameters here as well. 
\end{enumerate}  
Since the first choice is not applicable for the first Newton step as no update residual $R_{-1}$ is available yet, we use \eqref{eq:forcingC2} initially with the residuals $\lVert F(X_{j-1})\rVert_F,\lVert F(X_j)\rVert_F$ being available from the SCF-preprocessing. After the first Newton step, switching to \eqref{eq:forcingC1} lead to superior results in our experiments. The norms of $F(X_j)$ and the update residual are easily available from GL-GMRES and cheap to store. In addition to the safeguards mentioned above to avoid oversolving, we additionally limit $\eta_j$ to the interval $(0,0.9]$ in practice to avoid too inaccurate solutions.\\
The second key ingredient to inexact Newton methods is damping the inexact update directions to lead to a reduction in the residual. In the inexact Newton Backtracking method (INB) presented by Eisenstat and Walker \cite{forcingterm3}, the damping process is closely related to the choice of forcing terms. After an approximate update direction $\dx_j$ is computed, we damp $\dx_j$ by a parameter $\theta_{j}$ until the damped step satisfies 
$$\lVert F(X_j+\theta_j \dx_j)\rVert_F \leq \left[1-t(1-\eta_j)\right]\lVert F(X_j)\rVert_F.$$  
In each step, we choose the damping parameter to be the minimizer of a locally interpolating quadratic polynomial as follows: Let $g(\theta)=\lVert F(X_j+\theta \dx_j)\rVert_F^2$ and define by $p(\theta)$ the polynomial satisfying the Hermite interpolation conditions
$$p(0)=g(0),\quad p'(0)=g'(0),\quad p(1)=g(1).$$
One can easily verify that this polynomial is given by $p(\theta)=g(0) + g'(0)\theta + (g(1)-g(0)-g'(0))\theta^2$, where
\begin{eqnarray*}
	g(0)&=& \lVert F(X_k)\rVert_F^2,\quad g(1)\;=\; \lVert F(X_j+\dx_{j})\rVert_F^2,\\
	g'(\theta) 	&=& \frac{d}{d\theta} \left(\left\langle F(X_j+\theta \dx_{j}), F(X_k+\theta \dx_j)\right\rangle_F \right)\\
				&=& 2\left \langle L_F\left(X_j+\theta \dx_j,\frac{d}{d\theta} \theta \dx_{j}\right) , F(X_j+\theta \dx_{j})\right\rangle_F\\
				&=&	2\left\langle L_F(X_j+\theta \dx_{j},\dx_j),F(X_j+\theta \dx_j)\right\rangle_F,\\
	g'(0) &=& 2\left\langle L_F(X_j,\dx_j), F(X_j)\right\rangle_F.
\end{eqnarray*}
Thus, the optimal $\theta_{j}$ is given by the extreme point condition $p'(\theta^*)=0$, leading to 
$$\theta^{*} = \frac{-g'(0)}{2(g(1)-g(0)-g'(0))}.$$
The backtracking process is summed up in \cref{alg:backtrack}.\\ 
\begin{algorithm}[H]
		\SetAlgoVlined
	\KwIn{$X_j,\dx_j\in\RR^{n+k,k}$ Newton iterate and update direction, $F,L_F(X_j):\RR^{n+k,k}\rightarrow\RR^{n+k,k}$, $\eta_j\in[0,1)$ forcing term, $t>0$}
	\KwOut{Damped update direction $\dx_j$, damped forcing term $\eta_j$}
	Compute $g(0) = \lVert F(X_j)\rVert _F^2$, $g(1) = \lVert F(X_j+\dx_j)\rVert_F^2$ and set $s=g(1)-g(0)$\\ 		
	Compute $d=g'(0)=2\langle L_F(X_j,\dx_j),F(X_j)\rangle_F$\\ 
	\While{$\lVert F(X_j+\dx_j)\rVert _F > \left[1-t(1-\eta_j)\right]\lVert F(X_j)\rVert_F$}{
		Compute locally optimal $\theta = -\frac{d}{2(s-d)}$\\
		Compute damped parameters $\dx_j=\theta \dx_j$, $d = \theta d$ and $\eta_j=1-\theta (1-\eta)$
		} 
	\Return $\dx_j$ and $\eta_j$
	\vspace{-.2cm} 
	\caption{Inexact Newton backtracking}
	\label[algorithm]{alg:backtrack}
\end{algorithm}
Following the results from Eisenstat and Walker, we choose $t=10^{-4}$ in our experiments and allow for a limited number of backtracking steps only, where we found three or four steps to be an empirically good choice. 
To sum up our derivations, we will now briefly state the base variant of the inexact Matrix-Newton algorithm in \cref{alg:imn}. Note that the last step of SCF postprocessing is optional but allows for an easier comparison with the SCF solution. Computing an eigenvalue decomposition $\Lambda_{j+1}=Q^TDQ$ of $\Lambda$ instead and defining $V^*=V_{j+1}Q$, $\Lambda^*=D$, allows for comparable interpretability due to $D$ being diagonal while the computational cost is reduced since $k$ is usually a lot smaller than $n$.\\
\begin{algorithm}[H]
		\SetAlgoLined
	\KwIn{$V_0\in\OO_{n,k}(\RR)$, $G,H:\RR^{n,k}\rightarrow\RR^{n,n}$ and their Fréchet derivatives, $\tau>0$ convergence tolerance, $M\in\NN$ maximum iterations}
	\KwOut{$(V^*,\Lambda^*)$ solving \eqref{eq:NEPv}/\eqref{eq:GNEPv}}
	Choose a preprocessing tolerance and iteration limit and compute initial $X_0=\begin{bmatrix} V_0\\\Lambda_0\end{bmatrix}$ by SCF (\cref{alg:plainscf})\\
	Compute $\eta_0$ by \eqref{eq:forcingC2} using the residuals from \cref{alg:plainscf}\\
	\For{$j=0,\dots,M-1$}{
	 	Use global GMRES (\cref{alg:glgmres}) to find an approximate solution $E_j$ to $\lVert L_F(X_j,E_j)+F(X_j)\rVert_F\leq \eta_j \lVert F(X_j)\rVert _F$ \\
	 	Use backtracking (\cref{alg:backtrack}) to obtain a new damped iterate $X_{j+1} = X_j + \theta E_j$ and modified $\eta_j$ \\
	 	\If{$\lVert F(X_{j+1})\rVert _F<\tau$}{
	 		Tolerance reached, \textbf{break!}
	 	}
	 	Compute new forcing term $\eta_{j+1}$ by \eqref{eq:forcingC1}\\
	 } 
	 Perform a last postprocessing step of SCF using $X_{j+1}$ to obtain $(V^*,\Lambda^*)$, where now $\Lambda^*$ is diagonal\\
	 \Return $(V^*,\Lambda^*)$
	 \vspace{-.2cm} 
	\caption{Inexact matrix Newton algorithm for \eqref{eq:NEPv}/\eqref{eq:GNEPv} (JFNGK)}
	\label[algorithm]{alg:imn}
\end{algorithm}
\section{Numerical Experiments} \label{sec:numerical}
In this section, we apply our algorithm to a variety of test problems, namely the Kohn-Sham problem discussed in \cite{newtonvec} and \cite{optimalscf}, a robust LDA problem adapted from \cite{rlda} and a sum of trace ratios problem discussed in \cite{uinepv}. All experiments were performed using MATLAB 2024a on an Ubuntu 24.04 (64bit) system with AMD Ryzon 7 pro 4750U CPU (1.7 GHz).\\
In most of the experiments, we compare three different methods: 
\begin{enumerate}
\item The plain SCF algorithm (\cref{alg:plainscf}) using \texttt{eigs} to compute the $k$ desired eigenpairs.
\item The inexact Newton algorithm (\cref{alg:imn}) with SCF preprocessing and GL-GMRES as inner solver (JFNGK).
\item The inexact Newton algorithm (\cref{alg:imn}) with SCF preprocessing and MATLAB's built-in \texttt{gmres}-function as inner solver, where we solve \eqref{eq:newtonkron} without explicitly forming $K_F(x_j)$ but compute the matrix-vector product $K_F(\vect(X_j))v_i$ by $\vect(L_F(X_j,V_i))$ in a matrix-free way (JFNK).
\end{enumerate} 
In all our experiments, we use the exact Fréchet derivatives of $H(V)$ in all three methods as derived in the following sections. 
\subsection{The Kohn-Sham NEPv}\label[section]{sec:ks}
In this section, we consider two types of NEPv originating from Kohn-Sham theory, one of which is a simple model problem and one of which is the problem considered in \cite{newtonvec} for testing the performance of a vector-valued Newton approach. In the KS setting, the symmetric matrix function $H(V)$, often referred to as single-particle Hamiltonian in electronic structure calculations \cite{KSref2,KSref3,KSref1}, has the form
\begin{equation}
  H(V) = L + \Diag\left(L^{-1} \rho(V) - \gamma \rho(V)^{1/3}\right), \quad \gamma\geq 0, \label{eq:KS}
\end{equation}
where $L$ is a discrete Laplacian, $\rho(V)=\diag(VV^T)$ is the charge density of electrons, $\diag(X)$ and $\Diag(x)$ denote the vector containing the diagonal of the matrix $X$ and the diagonal matrix with the vector $x$ on its diagonal, respectively, and $\rho(V)^{1/3}$ is the elementwise cube root of the vector $\rho(V)$.\\
The NEPv with \eqref{eq:KS} can be derived by employing the KKT-conditions for the constrained energy minimization problem
\begin{equation*}
  \min_{V\in\OO_{n,k}(\RR)} \mathcal{E}(V) \define \frac{1}{2} \trace(V^T L V) + \frac{1}{4} \rho(V) L^{-1}\rho(V) - \frac{3}{4}\gamma\rho(V)^T \rho(V)^{1/3}.
\end{equation*}
Consequentially, we are interested in computing the eigenvectors associated with the $k$ smallest eigenvalues of $H(V)$.\\
For the computation of $L_H(V,E)$, the main effort is in finding the Fréchet derivative of $\rho(V)^{1/3}$, since $\Diag$ and $\diag$ are linear functions that have the simple Fréchet derivatives
$$ L_\Diag (x,e) = \Diag(e),\; x,e\in\RR^n \quad \text{and} \quad L_\diag(X,E)=\diag(E),\; X,E\in\RR^{n,n}.$$
By the product and chain rule, we then have that
$$L_\rho(V,E) = L_\diag(VV^T,L_{(XX^T)}(V,E))=\diag(VE^T + EV^T)= 2\diag(VE^T).$$
Therefor, one can verify that
$$L_{\rho^{1/3}}(V,E)=L_{x^{1/3}}(\rho(V),L_\rho(V,E))=\frac{2}{3} \rho(V)^{-2/3} \odot \diag(VE^T),$$
where $\rho(V)^{-2/3}$ is again applied elementwise and $\odot$ is the pointwise array multiplication. Note that $VV^T$ has only real positive diagonal entries, which are the squared norms of the corresponding rows of $V$ and, as such, $\rho(V)^{-2/3}$ is well defined.\\
By these parts, we obtain the Fréchet derivative
\begin{equation}
  L_H(V,E) = 2\Diag \left( L^{-1} \diag(VE^T) - \frac{1}{3}\gamma \left(\rho(V)^{-2/3}\odot \diag(VE^T)\right)\right). \label{eq:frechetks}
\end{equation}
\subsubsection{A simple 1D model} \label{sec:kssimple}
For our first test, consider the simple Kohn-Sham equation
\begin{equation}
	H(V)=L+\gamma \Diag(L^{-1} \rho(V)),\quad \gamma>0, \label{eq:KSsimple}
\end{equation}
where $L=\tridiag(-1,2,-1)$ is a discrete 1D-Laplacian. Using \eqref{eq:frechetks}, it is immediate to see that
\begin{equation}
  L_H(V,E) = 2\gamma \Diag\left(L^{-1} \diag(VE^T)\right). \label{eq:frechetkssimple}
\end{equation}
The dependency of \eqref{eq:KSsimple} on the choice of $\gamma$ and the convergence behaviour of SCF was analyzed in \cite{optimalscf}, where the main result was that SCF converges relatively fast for small choices of $\gamma$ and requires increasingly more steps for choices of $\gamma$ close to one. In fact, SCF failed to converge for $\gamma\geq 0.85$.\\
Note that the case $\gamma=0$ leads to the standard eigenvalue problem $LV=V\Lambda$, which is solved by SCF in one iteration, whereas larger choices of $\gamma$ lead to an increasing nonlinearity in the problem that makes solving the problem more challenging for SCF. In our experiment, we initialize the SCF algorithm with the two smallest eigenpairs of $L$ for different choices of $\gamma\neq 0$.\\
Doing so, we are able to reproduce the results observed in \cite[Ex. 1]{optimalscf} for the SCF algorithm, where the test problem $n=10$, $k=2$ and $\gamma\in\lbrace 0.5, 0.6, 0.7, 0.75, 0.8, 0.85, 0.9\rbrace$ are used. For the Newton method, we observe that the number of required iterations does not depend on the choice of $\gamma$ as much (see \cref{fig:kssimpleerr}).\\
We do two steps of SCF preprocessing before switching to Newton's method. The convergence tolerance of $\tau=\log((n+k)\cdot k)10^{-15}$ is reached by Newton's method for all choices of $\gamma$ after around nine to twelve steps while SCF only achieved the tolerance for $\gamma\leq 0.8$ and requires increasingly more steps as $\gamma$ tends to 1. As a consequene, we can also observe that the relative computation time required ($100\cdot\frac{t_{\scalebox{.5}{SCF}}}{t_{\scalebox{.5}{Newton}}}$) reduces from about $60\%$ for $\gamma=0.5$ to around $10\%$ for $\gamma=0.8$ and $1\%$ for the cases where SCF failed to converge within 4000 steps. 
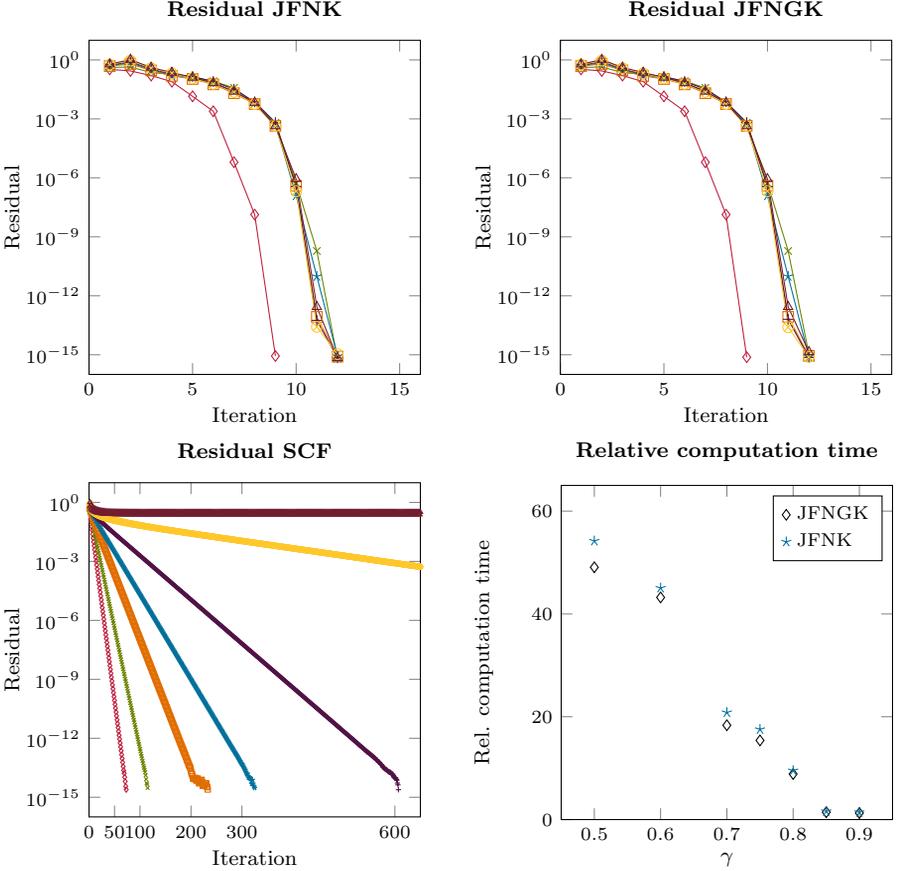
\begin{figure}[H]
\begin{subfigure}[l]{0.495\textwidth}
\begin{tikzpicture}
\begin{semilogyaxis}[
width=6cm,
height=6cm, 
xlabel style={yshift=.2cm},
xlabel={Iteration},
xmin=0, xmax=16,
xmode = linear,
xtick = {0,5,10,15},
x tick label style={font=\footnotesize},
ytickten = {-15,-12,-9,-6,-3,0},
ymin = 1e-16, ymax = 10,
ylabel style={yshift=-.2cm},
ylabel={Residual},
y tick label style={font=\footnotesize},
title style={font=\small\bfseries},
title={Residual JFNK},
legend style={font=\footnotesize},
legend pos=north east,
legend cell align={left},
label style={black,font=\small},
]
\addplot [color=tured, mark=diamond,mark size=2pt]
  table[row sep=crcr]{%
1	0.33233222397377\\
2	0.274910866115501\\
3	0.160133415319668\\
4	0.0780063832221844\\
5	0.0140862340131785\\
6	0.00244977216117655\\
7	6.26528304941312e-06\\
8	1.3584822667824e-08\\
9	8.67000840423086e-16\\
};

\addplot [color=tugreen, mark=x,mark size=2pt]
  table[row sep=crcr]{%
1	0.414941146843705\\
2	0.436984015278399\\
3	0.226809806743918\\
4	0.187280033627192\\
5	0.134735127911451\\
6	0.0825843917678137\\
7	0.0380292057583253\\
8	0.00605406929135051\\
9	0.000635220718683891\\
10	5.4996963131064e-07\\
11	1.9067856567705e-10\\
12	6.17678494026716e-16\\
};

\addplot [color=tuorange, mark=square,mark size=2pt]
  table[row sep=crcr]{%
1	0.487455640551286\\
2	0.621257257276362\\
3	0.291150653649281\\
4	0.153852095796874\\
5	0.107721211232756\\
6	0.057379815766812\\
7	0.0203643954082714\\
8	0.0056920040276397\\
9	0.000441426431698022\\
10	3.99397581542999e-07\\
11	8.41903303384718e-14\\
12	8.12650833956616e-16\\
};

\addplot [color=tublue, mark=star,mark size=2pt]
  table[row sep=crcr]{%
1	0.520622335945583\\
2	0.71942690436022\\
3	0.322149319117258\\
4	0.172724427378587\\
5	0.124931116006633\\
6	0.0686142484253638\\
7	0.0255748232002914\\
8	0.00653819740556698\\
9	0.000517302368864614\\
10	1.23798854577308e-07\\
11	9.74427824698245e-12\\
12	8.48363141713211e-16\\
};

\addplot [color=tupurple,mark=+,mark size=2pt]
  table[row sep=crcr]{%
1	0.552094692510583\\
2	0.820843384582362\\
3	0.352489327280474\\
4	0.191480628715575\\
5	0.142731297429587\\
6	0.080285853407252\\
7	0.0325798769929823\\
8	0.00734116161164006\\
9	0.000657255233086806\\
10	3.64330636169547e-07\\
11	5.63023561711409e-14\\
12	6.98454046386222e-16\\
};

\addplot [color=tuyellow, mark=otimes,mark size=2pt]
  table[row sep=crcr]{%
1	0.582088884498424\\
2	0.925109116824362\\
3	0.382319304345558\\
4	0.210133581402981\\
5	0.116702596864263\\
6	0.0605284250628155\\
7	0.0233948882966374\\
8	0.00559413510782873\\
9	0.000424418227439773\\
10	2.09904638472712e-07\\
11	2.63098920955193e-14\\
12	1.08513297249616e-15\\
};

\addplot [color=tudarkred, mark=triangle,mark size=2pt]
  table[row sep=crcr]{%
1	0.610774345713896\\
2	1.03191766395509\\
3	0.411791098289955\\
4	0.228696428025126\\
5	0.128269749824592\\
6	0.0674505664242036\\
7	0.0249310645428841\\
8	0.00670957017251702\\
9	0.000432004142981628\\
10	8.77235527476024e-07\\
11	2.75142054595091e-13\\
12	6.89573856283146e-16\\
};

\end{semilogyaxis}
\end{tikzpicture}
\end{subfigure}
\vspace{-.8cm}%
\begin{subfigure}[r]{0.495\textwidth}
\begin{tikzpicture}
\begin{semilogyaxis}[
width=6cm,
height=6cm, 
xlabel style={yshift=.2cm},
xlabel={Iteration},
xmin=0, xmax=16,
xmode = linear,
xtick = {0,5,10,15},
x tick label style={font=\footnotesize},
ytickten = {-15,-12,-9,-6,-3,0},
ymin = 1e-16, ymax = 10,
ylabel style={yshift=-.2cm},
ylabel={Residual},
y tick label style={font=\footnotesize},
title style={font=\small\bfseries},
title={Residual JFNGK},
legend style={font=\footnotesize},
legend pos=north east,
legend cell align={left},
label style={black,font=\small},
]
\addplot [color=tured, mark=diamond,mark size=2pt]
  table[row sep=crcr]{%
1	0.33233222397377\\
2	0.274910866115501\\
3	0.160133415319668\\
4	0.0780063832221844\\
5	0.0140862340131787\\
6	0.00244977216117655\\
7	6.26528304951676e-06\\
8	1.35848228585722e-08\\
9	7.55288322117785e-16\\
};

\addplot [color=tugreen, mark=x,mark size=2pt]
  table[row sep=crcr]{%
1	0.414941146843705\\
2	0.436984015278399\\
3	0.226809806743918\\
4	0.187280033627192\\
5	0.134735127911451\\
6	0.0825843917678135\\
7	0.0380292057583255\\
8	0.00605406929135008\\
9	0.000635220718683702\\
10	5.49969631499347e-07\\
11	1.90678705836702e-10\\
12	1.112562436566e-15\\
};

\addplot [color=tuorange, mark=square,mark size=2pt]
  table[row sep=crcr]{%
1	0.487455640551286\\
2	0.621257257276362\\
3	0.291150653649281\\
4	0.153852095796874\\
5	0.107721211232756\\
6	0.057379815766812\\
7	0.0203643954082725\\
8	0.00569200402764009\\
9	0.000441426431697704\\
10	3.99397581872241e-07\\
11	8.44492865061451e-14\\
12	8.83067967079767e-16\\
};

\addplot [color=tublue, mark=star,mark size=2pt]
  table[row sep=crcr]{%
1	0.520622335945583\\
2	0.71942690436022\\
3	0.322149319117258\\
4	0.172724427378586\\
5	0.124931116006633\\
6	0.0686142484253643\\
7	0.0255748232002914\\
8	0.00653819740556668\\
9	0.00051730236886469\\
10	1.23798855283522e-07\\
11	9.74416317538347e-12\\
12	8.0503132006544e-16\\
};

\addplot [color=tupurple,mark=+,mark size=2pt]
  table[row sep=crcr]{%
1	0.552094692510583\\
2	0.820843384582362\\
3	0.352489327280474\\
4	0.191480628715575\\
5	0.142731297429588\\
6	0.0802858534072517\\
7	0.0325798769929841\\
8	0.0073411616116386\\
9	0.000657255233086716\\
10	3.64330636935307e-07\\
11	6.38536456503077e-14\\
12	7.38136560039637e-16\\
};

\addplot [color=tuyellow, mark=otimes,mark size=2pt]
  table[row sep=crcr]{%
1	0.582088884498424\\
2	0.925109116824362\\
3	0.382319304345557\\
4	0.210133581402981\\
5	0.116702596864263\\
6	0.0605284250628152\\
7	0.0233948882966374\\
8	0.00559413510782898\\
9	0.000424418227439413\\
10	2.09904638800884e-07\\
11	2.50868106520752e-14\\
12	7.78765265973991e-16\\
};

\addplot [color=tudarkred, mark=triangle,mark size=2pt]
  table[row sep=crcr]{%
1	0.610774345713896\\
2	1.03191766395509\\
3	0.411791098289954\\
4	0.228696428025126\\
5	0.128269749824591\\
6	0.0674505664242037\\
7	0.0249310645428846\\
8	0.0067095701725178\\
9	0.000432004142981626\\
10	8.77235527373641e-07\\
11	2.75290238931693e-13\\
12	1.40117557262049e-15\\
};
\end{semilogyaxis}
\end{tikzpicture}
\end{subfigure}
\begin{subfigure}[l]{0.495\textwidth}
\include{figures/kssimple/scf_residual.tex}
\end{subfigure}
\vspace{-.9cm}
\begin{subfigure}[r]{0.495\textwidth}
\begin{tikzpicture}

\begin{axis}[%
width=6cm,
height=6cm, 
xlabel style={yshift=.2cm},
xlabel={$\gamma$},
xmin=0.45, xmax = 0.95,
xtick = {0.5,0.6,0.7,0.8,0.9},
x tick label style={font=\footnotesize},
ylabel style={yshift=-.2cm},
ylabel={Rel. computation time},
ymin = 0, ymax = 65,
ytick = {0, 20, 40, 60},
y tick label style={font=\footnotesize},
title style={font=\small\bfseries},
title={Relative computation time},
legend style={font=\footnotesize},
legend pos=north east,
legend cell align={left},
label style={black,font=\small},
]
\addplot [color=tublack, only marks, mark=diamond, mark size=2pt]
  table[row sep=crcr]{%
0.5	49.076312210517\\
0.6	43.2100404288147\\
0.7	18.3428388034036\\
0.75	15.3945507097435\\
0.8	8.85832587368849\\
0.85	1.43463559744035\\
0.9	1.27505196907765\\
};
\addlegendentry{JFNGK}

\addplot [color=tublue, only marks, mark=star, mark size=2pt]
  table[row sep=crcr]{%
0.5	54.1923541914336\\
0.6	45.0428828104935\\
0.7	20.8237999390919\\
0.75	17.5453624522606\\
0.8	9.52164303046067\\
0.85	1.59608643724264\\
0.9	1.40756388865239\\
};
\addlegendentry{JFNK}

\end{axis}

\end{tikzpicture}%
\end{subfigure}
\caption{Convergence results for simple KS-Problem by JFNK(top left), JFNGK(top right) and SCF(bottom left) as well as relative timing(bottom right). The convergence plots show results for $\gamma=0.5$(\textcolor{tured}{$\diamond$}), $\gamma=0.6$(\textcolor{tugreen}{x}), $\gamma=0.7$(\textcolor{tuorange}{$\square$}), $\gamma=0.75$(\textcolor{tublue}{$\ast$}), $\gamma=0.8$(\textcolor{tupurple}{+}), $\gamma=0.85$(\textcolor{tuyellow}{$\otimes$}) and $\gamma=0.9$(\textcolor{tudarkred}{$\triangle$}) }
\label{fig:kssimpleerr}
\end{figure}
\subsubsection{A 3D model}
For our second experiment, we consider \eqref{eq:KS} with parameters equivalent to those in \cite[Sec. 6]{newtonvec}, where $L$ is a discrete 3D-Laplacian and $\gamma=1$. That is, we have 
$$L=L_m\otimes I_{m^2} + I_m\otimes L_m\otimes I_m + I_{m^2} \otimes L_m,\quad m=\sqrt[3]{n},$$
with $L_m=\tridiag(-1,2,-1)\in\RR^{m,m}$ being a discrete 1D-Laplacian. Again, we start the SCF-algorithm with $V_0\in\OO_{n,k}(\RR)$ being the $k$ eigenvectors of $L$ associated with the $k$ smallest eigenvalues and set the maximum size of the Krylov subspace in both GMRES and GL-GMRES to 400. The convergence tolerance is set to $\tau=(n+k)10^{-15}$. In \autoref{fig:ksconvergence}, one can see the convergence results of SCF, JFNK and JFNGK for $n=32^3=32768$ and $k=2$ (left figure) as well as $k=8$ (right figure). For $k=2$, we switch to Newton as soon as $\lVert F(X_j)\rVert_F<10^{-5}$ in SCF, whereas for $k=8$ switching to Newton once $\lVert F(X_j)\rVert_F<10^{-6}$ proofed to be a better choice. Consequences of switching to Newton's method earlier or later are visualized in a second experiment (\autoref{fig:ksswitching}).
For the first experiment, JFNK and JFNGK can be seen to be mathematically equivalent for both $k=2$ as well as $k=8$, which is to be expected from \autoref{lem:gmres}. One can also see that, after switching to Newton, the convergence to the desired tolerance happens rapidly after only six steps for $k=2$, switching after 31 steps of SCF, and eleven steps for $k=8$, switching after 207 steps of SCF, while the SCF algorithm still requires about 70 and 350 more steps, respectively. The higher order of convergence can also be observed in the computation time, where JFNGK required $75\%$ of the SCF computation time and JFNK is only slightly slower at $77\%$ of the SCF computation time for $k=2$ as well as $71\%$ and $65\%$ for $k=8$, respectively. In this example, the average size of the Krylov subspace required in (GL-)GMRES for the inner solve was around 50 for $k=2$ and around 80 for $k=8$, which is only around $0.08\%$ and $0.03\%$ of the problem size ($(n+k)\cdot k$), respectively, allowing for a fast solve and significant improvement over the plain SCF. 
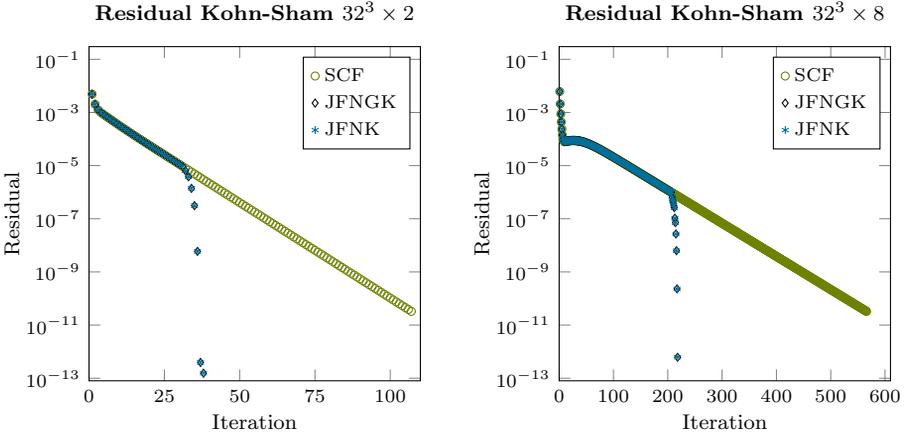
\begin{figure}[H]
\begin{subfigure}[l]{0.495\textwidth}
\begin{tikzpicture}
\begin{semilogyaxis}[
width=6cm,
height=6cm, 
xlabel style={yshift=.2cm},
xlabel={Iteration},
xmin=0, xmax=110,
xmode = linear,
xtick = {0,25,50,75,100},
x tick label style={font=\footnotesize},
ymin = 8e-14, ymax=0.3, 
ytickten = {-13,-11,-9,-7,-5,-3,-1},
ylabel style={yshift=-.2cm},
ylabel={Residual},
y tick label style={font=\footnotesize},
title style={font=\small\bfseries},
title={Residual Kohn-Sham $32^3\times 2$},
legend style={font=\footnotesize},
legend pos=north east,
legend cell align={left},
label style={black,font=\small},
]
\addplot [color=tugreen, only marks, mark=o, mark size = 1.5pt]
  table[row sep=crcr]{%
1	0.00486921102840518\\
2	0.00203231913941258\\
3	0.00125491349090554\\
4	0.000982441297012943\\
5	0.000812240485508377\\
6	0.000676501522406266\\
7	0.000563970288567445\\
8	0.000470557942154877\\
9	0.000393128961116848\\
10	0.000328947038253841\\
11	0.0002756786067012\\
12	0.000231383481212338\\
13	0.000194471197857672\\
14	0.000163645709468688\\
15	0.00013785200869247\\
16	0.000116229891972372\\
17	9.80758776953331e-05\\
18	8.2812587478678e-05\\
19	6.99643846191507e-05\\
20	5.91380477632812e-05\\
21	5.00074185230282e-05\\
22	4.2301160556062e-05\\
23	3.57929512518633e-05\\
24	3.0293579717784e-05\\
25	2.5644545447068e-05\\
26	2.1712845096247e-05\\
27	1.83867057898349e-05\\
28	1.55720772871313e-05\\
29	1.31897363026931e-05\\
30	1.11728875117286e-05\\
31	9.46516973629328e-06\\
32	8.01899426036843e-06\\
33	6.79415661670185e-06\\
34	5.7566744197964e-06\\
35	4.87781266746821e-06\\
36	4.13326502989067e-06\\
37	3.50246524111356e-06\\
38	2.96800725600284e-06\\
39	2.51515652008888e-06\\
40	2.13143771192955e-06\\
41	1.80628672245827e-06\\
42	1.53075670211728e-06\\
43	1.29726964665763e-06\\
44	1.09940636649777e-06\\
45	9.31728857413065e-07\\
46	7.89630016461437e-07\\
47	6.69206465332914e-07\\
48	5.67150893603514e-07\\
49	4.80660926897727e-07\\
50	4.07361953782946e-07\\
51	3.45241771967485e-07\\
52	2.92595232596604e-07\\
53	2.47977362264932e-07\\
54	2.10163635901835e-07\\
55	1.78116331363345e-07\\
56	1.50956022714277e-07\\
57	1.27937424302338e-07\\
58	1.08428923678778e-07\\
59	9.18952400753409e-08\\
60	7.78827290700641e-08\\
61	6.60069320400595e-08\\
62	5.59420173651856e-08\\
63	4.74118471169958e-08\\
64	4.0182390065786e-08\\
65	3.40553042826651e-08\\
66	2.88624958779304e-08\\
67	2.44614999915216e-08\\
68	2.07315774952029e-08\\
69	1.75704019879585e-08\\
70	1.4891248138586e-08\\
71	1.26206148924617e-08\\
72	1.06962119447018e-08\\
73	9.06524337455177e-09\\
74	7.6829679999168e-09\\
75	6.51146331425802e-09\\
76	5.51858964482317e-09\\
77	4.67711189362636e-09\\
78	3.96394283142407e-09\\
79	3.35951751295681e-09\\
80	2.84725596765475e-09\\
81	2.41310555140636e-09\\
82	2.04515308291626e-09\\
83	1.73330757053133e-09\\
84	1.46901215654182e-09\\
85	1.24501626647787e-09\\
86	1.05517488880869e-09\\
87	8.94281387217842e-10\\
88	7.57920591875618e-10\\
89	6.42352137830978e-10\\
90	5.44407406343131e-10\\
91	4.61394872554026e-10\\
92	3.91040655775294e-10\\
93	3.31414235450384e-10\\
94	2.80880389152855e-10\\
95	2.38052765687583e-10\\
96	2.01753259338512e-10\\
97	1.70990394742152e-10\\
98	1.44916995454343e-10\\
99	1.22820657300079e-10\\
100	1.04092172033612e-10\\
101	8.82205852713019e-11\\
102	7.47685687344317e-11\\
103	6.33675335014644e-11\\
104	5.37057613898229e-11\\
105	4.55162580773503e-11\\
106	3.85763427217586e-11\\
107	3.26945014007273e-11\\
};

\addplot [color=tublack, only marks, mark=diamond, mark size=1.5pt]
  table[row sep=crcr]{%
1	0.00486921102840518\\
2	0.00203231913941258\\
3	0.00125491349090554\\
4	0.000982441297012943\\
5	0.000812240485508377\\
6	0.000676501522406266\\
7	0.000563970288567445\\
8	0.000470557942154877\\
9	0.000393128961116848\\
10	0.000328947038253841\\
11	0.0002756786067012\\
12	0.000231383481212338\\
13	0.000194471197857672\\
14	0.000163645709468688\\
15	0.00013785200869247\\
16	0.000116229891972372\\
17	9.80758776953331e-05\\
18	8.2812587478678e-05\\
19	6.99643846191507e-05\\
20	5.91380477632812e-05\\
21	5.00074185230282e-05\\
22	4.2301160556062e-05\\
23	3.57929512518633e-05\\
24	3.0293579717784e-05\\
25	2.5644545447068e-05\\
26	2.1712845096247e-05\\
27	1.83867057898349e-05\\
28	1.55720772871313e-05\\
29	1.31897363026931e-05\\
30	1.11728875117286e-05\\
31	9.46516973629328e-06\\
32	6.37779360077941e-06\\
33	3.78970737463785e-06\\
34	1.39328206486564e-06\\
35	3.11643916639308e-07\\
36	5.93361743158572e-09\\
37	3.9184735530483e-13\\
38	1.54450443853292e-13\\
};

\addplot [color=tublue, only marks, mark=asterisk, mark size=1.5pt]
  table[row sep=crcr]{%
1	0.00486921102840518\\
2	0.00203231913941258\\
3	0.00125491349090554\\
4	0.000982441297012943\\
5	0.000812240485508377\\
6	0.000676501522406266\\
7	0.000563970288567445\\
8	0.000470557942154877\\
9	0.000393128961116848\\
10	0.000328947038253841\\
11	0.0002756786067012\\
12	0.000231383481212338\\
13	0.000194471197857672\\
14	0.000163645709468688\\
15	0.00013785200869247\\
16	0.000116229891972372\\
17	9.80758776953331e-05\\
18	8.2812587478678e-05\\
19	6.99643846191507e-05\\
20	5.91380477632812e-05\\
21	5.00074185230282e-05\\
22	4.2301160556062e-05\\
23	3.57929512518633e-05\\
24	3.0293579717784e-05\\
25	2.5644545447068e-05\\
26	2.1712845096247e-05\\
27	1.83867057898349e-05\\
28	1.55720772871313e-05\\
29	1.31897363026931e-05\\
30	1.11728875117286e-05\\
31	9.46516973629328e-06\\
32	6.37779360078571e-06\\
33	3.78970737740026e-06\\
34	1.39328218356968e-06\\
35	3.11643871272076e-07\\
36	5.93361494579922e-09\\
37	3.92557288862086e-13\\
38	1.54598490084917e-13\\
};
\legend{SCF,JFNGK,JFNK}
\end{semilogyaxis}
\end{tikzpicture}%
\end{subfigure}
\vspace{-1cm}%
\begin{subfigure}[r]{0.495\textwidth}
\include{figures/ks32768x8e-6/residuals.tex}
\end{subfigure}
	\caption{Convergence results for KS-problem, $n=32^3$, $\gamma=1$, $k=2$(left) and $k=8$(right)}
	\label{fig:ksconvergence}
\end{figure}     
In a second experiment, we consider again the KS-problem of size $n=32^3$, $k=8$, and discuss the importance of switching to Newton at the correct time. In our setup before, we use $10^{-6}$ as our criterion for the switch to Newton's method. Now, we want to see the behavior of JFNGK if we switch earlier, at $\lVert F(X_j)\rVert_F<10^{-5}$, which is the tolerance used for $k=2$, or a lot later, at $\lVert F(X_j)\rVert_F<10^{-8}$. The resulting convergence curves are displayed in \autoref{fig:ksswitching}. One can see that 
\begin{figure}[H]
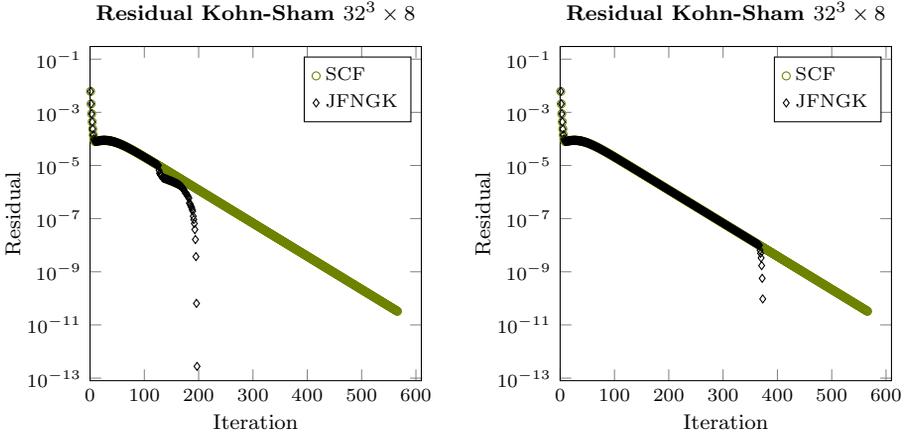

\begin{subfigure}[l]{0.495\textwidth}
\include{figures/ks32768x8e-5/residuals.tex}
\end{subfigure}
\vspace{-1cm}%
\begin{subfigure}[r]{0.495\textwidth}
\include{figures/ks32768x8e-8/residuals.tex}
\end{subfigure}
	\caption{Consequences of switching to Newton too early(left) or too late(right)}
	\label{fig:ksswitching}
\end{figure}
switching too early (left plot) can heavily delay the rapid convergence, which will come into play after around 190 steps, when 60 steps of Newton's method have already been taken. Note that the residual is around $10^{-6}$ at that point, which is exactly the tolerance we picked in our first experiment. While JFNGK converges after 198 steps, which is nine steps faster than in the first test, it took almost twice as long as in \autoref{fig:ksconvergence} since the additional 60 steps of Newton's method are a lot more costly than 80 more steps of SCF.\\
On the other hand, switching late naturally leads to an even more rapid convergence than in the first experiment, leading to only seven Newton steps required. However, the switch to Newton's method happens around 150 SCF steps later, resulting in a slower computation of around $95\%$ of SCF, compared to $71\%$ in the first experiment.\\
The main takeaway from this experiment is that picking the correct timing for the switch to Newton's method is key to the performance of JFNGK and can heavily depend on the problem. Since switching too early leads a more severe slowdown than switching a bit too late, we usually choose a rather conservative tolerance for the switch to make sure to be close enough to benefit from the higher order of convergence immediately.\\
Since the choice of tolerance is mostly empirical in our experiments, finding a flexible and easy to check criterion for the switch is an interesting point of research and could lead to a significant performance and reliability improvement.
\subsection{The robust LDA GNEPv}
For our next experiment, we consider the GNEPv
\begin{equation}
  H(v) v = \lambda G(v) v \label{eq:rlda}
\end{equation}
originating from robust Rayleigh quotient optimization (RRQM), where the the goal is to find a vector $v\in\RR^n\setminus\lbrace0\rbrace$ minimizing the Rayleigh quotient
$$r(v)=\frac{v^T H(v) v}{v^T G(v) v}.$$
Finding robust solutions can be interpreted as optimizing an uncertain, data-based problem for the worst possible behavior of the input data \cite{rlda}.\\
In our experiment, we discuss a robust LDA-problem, where we try to find the Fisher discriminant vector $v$ in a two set classification problem. If $X,Y\in\RR^n$ are two random variables with mean $\mu_X,\mu_Y$ and covariance matrices $\Sigma_X,\Sigma_Y$, it is easy to see that the discriminant vector $v^*$ leading to an optimal separation can, up to a scale factor, be uniquely determined by
\begin{equation}
v^*= (\Sigma_X + \Sigma_Y)^{-1} (\mu_Y-\mu_Y). \label{eq:CLDA}
\end{equation}
However, in real life problems, one has usually access to a sample of data only and the sample means and covariances are empirically estimated rather than known explicitly. That is, we have uncertainty in the problem parameters and try to optimize for the worst case behaviour.\\
In \cite[Ex. 5]{rlda}, the following problem is considered: Let
$$H(v)\equiv\overline{\Sigma}_X + \overline{\Sigma}_Y + (\delta_X + \delta_Y)I_n\quad \text{and}\quad G(v)=f(v)f(v)^T,$$
where
\begin{equation}
  f(v) = (\overline{\mu}_X - \overline{\mu}_Y) - \sign\left(v^T (\overline{\mu}_X - \overline{\mu}_Y)\right) \left(\frac{S_X v}{\sqrt{v^T S_X v}}+\frac{S_Y v}{\sqrt{v^T S_Y v}}\right).
\end{equation}
Here, $\overline{\mu}_X,\overline{\mu}_Y$ and $\overline{\Sigma}_X,\overline{\Sigma}_Y$ are the estimated sample means and covariance matrices, respectively. They are estimated by randomly sampling the data 100 times and computing the average of all 100 computed sample means and covariances. The parameters $\delta_X,\delta_Y$ describe the maximum deviation to the average covariance matrices, that is
$$\delta_X = \max_{i=1,...,100} \lVert \overline{\Sigma}_X - \Sigma_X^{(i)}\rVert _F,$$
where $\Sigma_X^{(i)}$ is the sample covariance matrix in the $i$-th resampling iteration. Similarly, the covariance matrices of the sampled mean values $\mu_X^{(i)},\mu_Y^{(i)}$ are used to compute $S_X$ and $S_Y$.\\
Note than in this application, the Fréchet derivative of $G$ has the special low-rank structure $$L_G(v,e) = f(v) L_f(v,e)^T + L_f(v,e) f(v)^T,$$ meaning that we do not want to compute the full size matrix explicitly but use scalar products with $f(v)$ and $L_f(v,e)$ instead. The Fréchet derivative of $f(v)$ can be assembled from the Fréchet derivatives of the quotients $q_S(v) = \frac{Sv}{\sqrt{v^TSv}}$, which are given by
$$L_{q_S}(v,e) = S\left(\frac{1}{\sqrt{v^TSv}}e-\frac{v^TS e}{\sqrt{(v^TSv)^3}}v\right).$$
Note that $\sign(x)$ is only differentiable with $\sign'(x)=0$ if $x\neq 0$, so we require $v$ to not be orthogonal to the difference in sample means for $f$ to be differentiable. In that case, we have
$$ L_f(v,e) = -\sign\left(v^T (\overline{\mu}_X - \overline{\mu}_Y)\right) \left(L_{q_{S_X}}(v,e) + L_{q_{S_Y}}(v,e)\right).$$
In our experiment, we compare the performance of our algorithm to the two different SCF algorithms provided by Bai, Lu and Vandereycken in the RobustRQ software package, which is available from \url{http://www.unige.ch/~dlu/rbstrq.html}. Their level-shifted SCF-algorithm of order one is denoted by SCF in our plots while their second order SCF-algorithm, which converges quadratically, is denoted by SCF2. Additionally, we compare the classification accuracy to the non-robust LDA solution using the same estimated mean and covariances in \eqref{eq:CLDA}.\\
To set up the classification, we split the data into training- and test-data, where the training-data is used to estimate the problem parameters $\overline{\Sigma}$, $\overline{\mu}$, $S$ and $\delta$ and the test-data is used to verify the classification accuracy on unknown data. The amount of training points used is controlled by a parameter $\alpha\in(0,1)$. For every choice of $\alpha$, we perform 100 tests with the data being split randomly every time and plot the average test sample accuracy (TSA) and its deviation.\\
We verify the performance of our algorithm on two benchmark problems from the UCI machine learning repository \cite{uci}: The  ionosphere data set, which consists of 351 instances with 34 attributes each, and the sonar data set, which consists of of 208 instances with 60 attributes each. Thus, the corresponding GNEPv will be of size $n=34$ and $n=60$, respectively. Since $k=1$, we use the build-in GMRES for our experiments.\\
For the ionosphere set, we use $\alpha\in\lbrace 0.1, 0.15, 0.2, 0.4, 0.6,0.7\rbrace$ for the data splitting and preprocess by 100 steps of SCF and 2 steps of SCF2 or until a tolerance of $10^{-3}$ is reached and switch to Newton's method afterwards. Here, we allow a Krylov space of size 30 and use $\tau=10^{-10}$ as our convergence tolerance.\\
In \cref{fig:rldaion}, one can see that the Newton method leads to classifications of comparable accuracy to SCF and SCF2 in a similar computation time. Only for $\alpha=0.1$, the accuracy by JFNK is a bit worse.  All three approaches are superior to the CLDA in the sense of classification accuracy, justifying the idea to use robust classification. For smaller choices of $\alpha$, Newton's method is faster than SCF but slower than SCF2, while for $\alpha\geq0.4$, it is faster than both SCF and SCF2. 
For the sonar data set, we use $\alpha\in\lbrace 0.3, 0.4, 0.5, 0.6, 0.7, 0.8\rbrace$ and preprocess by 60 steps of SCF and 6 steps of SCF or until a tolerance of $10^{-3}$ is reached. Again, we allow 30 GMRES-steps and use $\tau=10^{-10}$. In \cref{fig:rldasonar}, one can observe that the TSA of Newton's method is able to match SCF and SCF2 while all three robust methods outperform the CLDA. In terms of computation time, our algorithm is faster than SCF but slower than SCF2 for $\alpha\leq 0.4$ and marginally faster than both algorithms for $\alpha\geq 0.5$.\\
Overall, we can conclude that the Newton algorithm is able to outperform or at least match the SCF algorithm for both data sets and all choices of $\alpha$ while being approximately on par with or slightly faster than the SCF2 algorithm for medium and large choices of $\alpha$ in both examples. Thus, our algorithm is able to compete with SCF for GNEPv \eqref{eq:GNEPv} of this type as well.      
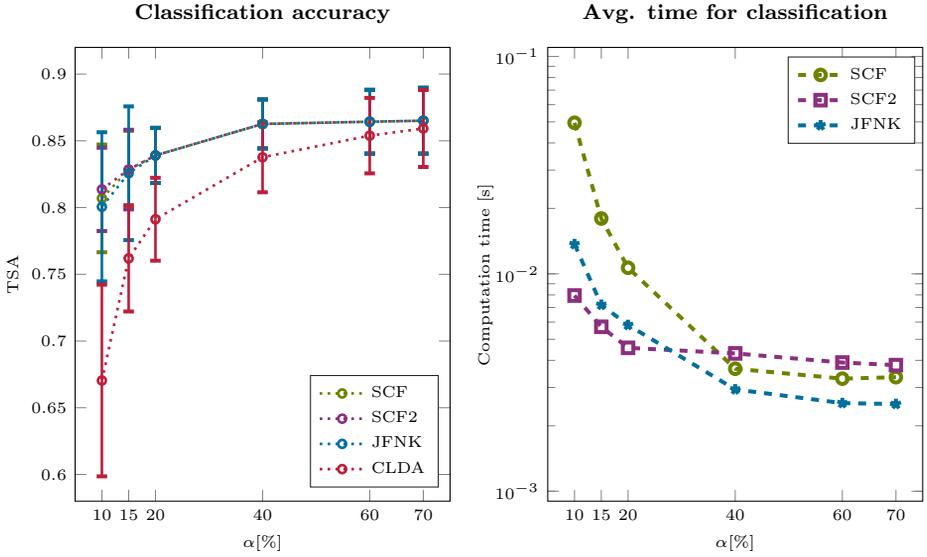
\begin{figure}[H]
\begin{subfigure}[l]{0.495\textwidth}
\begin{tikzpicture}
\begin{axis}[%
width=5cm,
height=6cm,
scale only axis,
xmin=5,
xmax=75,
xlabel={$\alpha\text{[\%]}$},
xtick = {10, 15, 20, 40, 60, 70},
ymin=0.58,
ymax=0.92,
ylabel={TSA},
ytick = {0.6, 0.65, 0.7, 0.75, 0.8, 0.85, 0.9},
tick label style = {font=\scriptsize},
label style = {black, font=\scriptsize},
x label style={yshift=.2cm},
y label style={yshift=-.4cm},
axis background/.style={fill=white},
legend style={at={(0.97,0.03)}, anchor=south east, legend cell align=left, align=left, draw=white!15!black, font=\scriptsize},
title style={font=\small\bfseries},
title={Classification accuracy},
]
\addplot [color=tugreen, dotted, line width=1pt, mark=o, mark options={solid, tugreen}, mark size = 1.5pt]
 plot [error bars/.cd, y dir = both, y explicit, error bar style={line width=1pt, solid}, error mark options={solid, rotate = 90, tugreen, mark size=2pt,line width=1.5pt}]
 table[row sep=crcr, y error plus index=2, y error minus index=3]{%
10	0.806898734177215	0.0403735056195196	0.0403735056195196\\
15	0.828361204013378	0.0291482982288973	0.0291482982288973\\
20	0.839039145907473	0.0206880820512772	0.0206880820512772\\
40	0.862654028436018	0.0183489427132093	0.0183489427132093\\
60	0.864326241134753	0.0238898323805494	0.0238898323805494\\
70	0.865094339622641	0.0246700978344484	0.0246700978344484\\
};
\addlegendentry{SCF}

\addplot [color=tulightpurple, dotted, line width=1pt, mark=o, mark options={solid, tulightpurple}, mark size = 1.5pt]
 plot [error bars/.cd, y dir = both, y explicit, error bar style={line width=1pt, solid}, error mark options={solid, rotate = 90, tulightpurple, mark size=2pt,line width=1.5pt}]
 table[row sep=crcr, y error plus index=2, y error minus index=3]{%
10	0.813607594936709	0.0312094354993918	0.0312094354993918\\
15	0.828461538461538	0.0298583602244609	0.0298583602244609\\
20	0.839074733096085	0.0206368504404832	0.0206368504404832\\
40	0.862654028436018	0.0183489427132093	0.0183489427132093\\
60	0.864326241134753	0.0238898323805494	0.0238898323805494\\
70	0.865094339622641	0.0246700978344484	0.0246700978344484\\
};
\addlegendentry{SCF2}

\addplot [color=tublue, dotted, line width=1pt, mark=o, mark options={solid, tublue}, mark size = 1.5pt]
 plot [error bars/.cd, y dir = both, y explicit, error bar style={line width=1pt, solid}, error mark options={solid, rotate = 90, tublue, mark size=2pt,line width=1.5pt}]
 table[row sep=crcr, y error plus index=2, y error minus index=3]{%
10	0.800569620253164	0.0558646008235972	0.0558646008235972\\
15	0.82571906354515	0.0500914825236274	0.0500914825236274\\
20	0.839074733096085	0.0206368504404832	0.0206368504404832\\
40	0.862654028436018	0.0183489427132093	0.0183489427132093\\
60	0.864326241134753	0.0238898323805494	0.0238898323805494\\
70	0.865094339622641	0.0246700978344484	0.0246700978344484\\
};
\addlegendentry{JFNK}

\addplot [color=tured, dotted, line width=1pt, mark=o, mark options={solid, tured}, mark size = 1.5pt]
 plot [error bars/.cd, y dir = both, y explicit, error bar style={line width=1pt, solid}, error mark options={solid, rotate = 90, tured, mark size=2pt,line width=1.5pt}]
 table[row sep=crcr, y error plus index=2, y error minus index=3]{%
10	0.670411392405063	0.0718291342061486	0.0718291342061486\\
15	0.761939799331104	0.039818304477345	0.039818304477345\\
20	0.791245551601424	0.0310498285014945	0.0310498285014945\\
40	0.837677725118483	0.0262398535878169	0.0262398535878169\\
60	0.853900709219859	0.0282791045844503	0.0282791045844503\\
70	0.859245283018868	0.0288424334165034	0.0288424334165034\\
};
\addlegendentry{CLDA}

\end{axis}
\end{tikzpicture}%
\end{subfigure}
\vspace{-.8cm}%
\begin{subfigure}[r]{0.495\textwidth}
\begin{tikzpicture}

\begin{axis}[%
width=5cm,
height=6cm,
scale only axis,
xmin=5,
xmax=75,
xlabel={$\alpha\text{[\%]}$},
ylabel={Computation time [s]},
xtick = {10, 15, 20, 40, 60, 70},
ymode=log,
ymin=0.0009,
ymax=0.11,
ytickten = {-3,-2,-1},
title style={font=\small\bfseries},
title={Avg. time for classification},
tick label style = {font=\scriptsize},
label style = {black, font=\scriptsize},
x label style={yshift=.2cm},
y label style={yshift=-.4cm},
axis background/.style={fill=white},
legend style={legend cell align=left, align=left, draw=white!15!black,font=\scriptsize}
]
\addplot [color=tugreen, dashed, line width=1.5pt, mark=o, mark options={solid, tugreen}]
  table[row sep=crcr]{%
10	0.04952116\\
15	0.01796865\\
20	0.01066285\\
40	0.00366018\\
60	0.00329516\\
70	0.00334686\\
};
\addlegendentry{SCF}

\addplot [color=tulightpurple, dashed, line width=1.5pt, mark=square, mark options={solid, tulightpurple}]
  table[row sep=crcr]{%
10	0.0079485\\
15	0.00570704\\
20	0.00456722\\
40	0.00430531\\
60	0.00390907\\
70	0.00380214\\
};
\addlegendentry{SCF2}

\addplot [color=tublue, dashed, line width=1.5pt, mark=star, mark options={solid, tublue}]
  table[row sep=crcr]{%
10	0.01368858\\
15	0.00720078\\
20	0.00581416\\
40	0.00293983\\
60	0.00254267\\
70	0.00251844\\
};
\addlegendentry{JFNK}

\end{axis}
\end{tikzpicture}%
\end{subfigure}
\caption{Convergence results for different choices of $\alpha$, Ionosphere-Data-Set}
\label{fig:rldaion}
\end{figure}
\vspace*{-.5cm}%
\begin{figure}[H]
\begin{subfigure}[l]{0.495\textwidth}
\begin{tikzpicture}
\begin{axis}[%
width=5cm,
height=6cm,
scale only axis,
xmin=25,
xmax=85,
xlabel={$\alpha\text{[\%]}$},
xtick = {30, 40, 50, 60, 70,80},
ymin=0.48,
ymax=0.82,
ylabel={TSA},
ytick = {0.5,0.55, 0.6,0.65, 0.7, 0.75,0.8},
tick label style = {font=\scriptsize},
label style = {black, font=\scriptsize},
x label style={yshift=.2cm},
y label style={yshift=-.4cm},
axis background/.style={fill=white},
legend style={at={(0.97,0.03)}, anchor=south east, legend cell align=left, align=left, draw=white!15!black, font=\scriptsize},
title style={font=\small\bfseries},
title={Classification accuracy},
]
\addplot [color=tugreen, dotted, line width=1pt, mark=o, mark options={solid, tugreen}, mark size = 1.5pt]
 plot [error bars/.cd, y dir = both, y explicit, error bar style={line width=1pt, solid}, error mark options={solid, rotate = 90, tugreen, mark size=2pt,line width=1.5pt}]
 table[row sep=crcr, y error plus index=2, y error minus index=3]{%
30	0.700890410958904	0.0451722899156914	0.0451722899156914\\
40	0.7168	0.038626938838593	0.038626938838593\\
50	0.732019230769231	0.0389838035651989	0.0389838035651989\\
60	0.737261904761905	0.040086508014359	0.040086508014359\\
70	0.755238095238095	0.0481336102245255	0.0481336102245255\\
80	0.747857142857143	0.0627750823537531	0.0627750823537531\\
};
\addlegendentry{SCF}

\addplot [color=tulightpurple, dotted, line width=1pt, mark=o, mark options={solid, tulightpurple}, mark size = 1.5pt]
 plot [error bars/.cd, y dir = both, y explicit, error bar style={line width=1pt, solid}, error mark options={solid, rotate = 90, tulightpurple, mark size=2pt,line width=1.5pt}]
 table[row sep=crcr, y error plus index=2, y error minus index=3]{%
30	0.70027397260274	0.0452044311823144	0.0452044311823144\\
40	0.7168	0.038626938838593	0.038626938838593\\
50	0.732019230769231	0.0389838035651989	0.0389838035651989\\
60	0.737261904761905	0.040086508014359	0.040086508014359\\
70	0.755238095238095	0.0481336102245255	0.0481336102245255\\
80	0.747857142857143	0.0627750823537531	0.0627750823537531\\
};
\addlegendentry{SCF2}

\addplot [color=tublue, dotted, line width=1pt, mark=o, mark options={solid, tublue}, mark size = 1.5pt]
 plot [error bars/.cd, y dir = both, y explicit, error bar style={line width=1pt, solid}, error mark options={solid, rotate = 90, tublue, mark size=2pt,line width=1.5pt}]
 table[row sep=crcr, y error plus index=2, y error minus index=3]{%
30	0.70027397260274	0.0452044311823144	0.0452044311823144\\
40	0.7168	0.038626938838593	0.038626938838593\\
50	0.732019230769231	0.0389838035651989	0.0389838035651989\\
60	0.737261904761905	0.040086508014359	0.040086508014359\\
70	0.755238095238095	0.0481336102245255	0.0481336102245255\\
80	0.747857142857143	0.0627750823537531	0.0627750823537531\\
};
\addlegendentry{JFNK}

\addplot [color=tured, dotted, line width=1pt, mark=o, mark options={solid, tured}, mark size = 1.5pt]
 plot [error bars/.cd, y dir = both, y explicit, error bar style={line width=1pt, solid}, error mark options={solid, rotate = 90, tured, mark size=2pt,line width=1.5pt}]
 table[row sep=crcr, y error plus index=2, y error minus index=3]{%
30	0.568287671232877	0.0626272963419186	0.0626272963419186\\
40	0.668	0.0495311349893512	0.0495311349893512\\
50	0.693557692307692	0.0427542085033657	0.0427542085033657\\
60	0.711904761904762	0.0509591010299667	0.0509591010299667\\
70	0.735873015873016	0.0500689848872931	0.0500689848872931\\
80	0.730476190476191	0.0645546012941675	0.0645546012941675\\
};
\addlegendentry{CLDA}

\end{axis}
\end{tikzpicture}%
\end{subfigure}
\vspace{-.8cm}%
\begin{subfigure}[r]{0.495\textwidth}
\begin{tikzpicture}

\begin{axis}[%
width=5cm,
height=6cm,
scale only axis,
xmin=25,
xmax=85,
xlabel={$\alpha\text{[\%]}$},
ylabel={Computation time [s]},
xtick = {30,40, 50,60,70,80},
ymode=log,
ymin=0.009,
ymax=0.5,
ytickten = {-2,-1,0},
title style={font=\small\bfseries},
title={Avg. time for classification},
tick label style = {font=\scriptsize},
label style = {black, font=\scriptsize},
x label style={yshift=.2cm},
y label style={yshift=-.4cm},
axis background/.style={fill=white},
legend style={legend cell align=left, align=left, draw=white!15!black,font=\scriptsize}
]
\addplot [color=tugreen, dashed, line width=1.5pt, mark=o, mark options={solid, tugreen}]
  table[row sep=crcr]{%
30	0.18429974\\
40	0.1309971\\
50	0.09248391\\
60	0.06041457\\
70	0.0391535\\
80	0.02431865\\
};
\addlegendentry{SCF}

\addplot [color=tulightpurple, dashed, line width=1.5pt, mark=square, mark options={solid, tulightpurple}]
 table[row sep=crcr]{%
30	0.01700054\\
40	0.01584523\\
50	0.01587898\\
60	0.01641482\\
70	0.01488733\\
80	0.01432658\\
};
\addlegendentry{SCF2}

\addplot [color=tublue, dashed, line width=1.5pt, mark=star, mark options={solid, tublue}]
  table[row sep=crcr]{%
30	0.03843434\\
40	0.01547881\\
50	0.01471958\\
60	0.01444936\\
70	0.01361222\\
80	0.01195524\\
};
\addlegendentry{JFNK}

\end{axis}
\end{tikzpicture}%
\end{subfigure}
\caption{Convergence results for different choices of $\alpha$, Sonar-Data-Set}
\label{fig:rldasonar}
\end{figure}
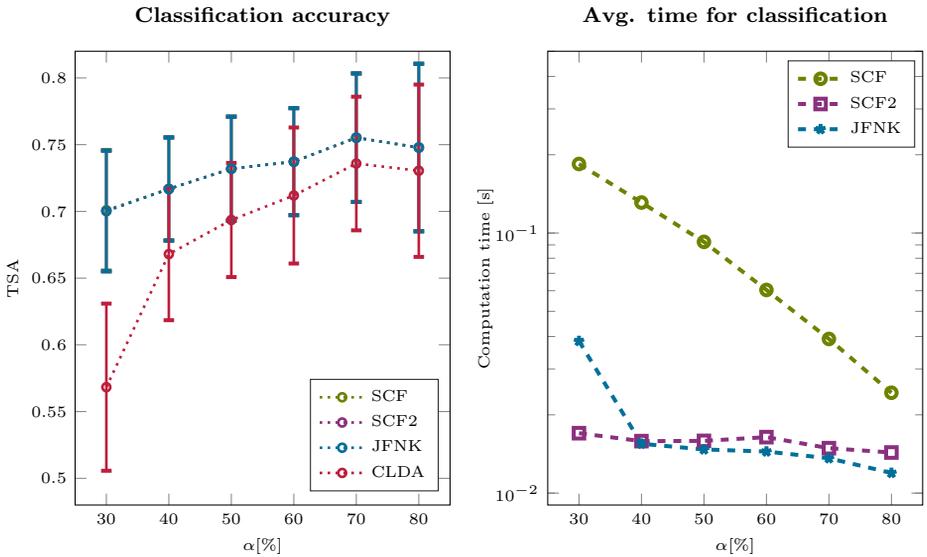

\subsection{The sum of trace ratios}
As our last example, we  consider the sum of trace ratios (SumTR) problem from \cite{uinepv}, where we want to maximize 
\begin{equation}
f_\alpha(V)\define (1-\alpha)\phi(V) + \alpha \psi(V)\trace(V^TD),\quad \alpha\in[0,1],~D\in\RR^{n,k}, \label{eq:sumTR}
\end{equation} 
subject to $V^TV=I_k$. In Example 7.1 and 7.2, the specific problem 
$$\phi(X)\define\frac{\trace(V^TAV)}{\trace(V^TBV)},\quad \psi(V)=\frac{1}{\sqrt{\trace(V^TBV)}},$$
where $A,B\in\RR^{n,n}$ are s.p.d., is considered. Note that the case $\alpha=0$ gives the well-known Trace Ratio Maximization problem (TRP), which is a special type of NEPv for which SCF is known to converge quadratically\cite{nepv}.\\
In \cite{uinepv}, it has been shown that the maximization of $f_\alpha(V)$ from \eqref{eq:sumTR} is equivalent to solving the NEPv 
$$ H_\alpha(V)V=V\Lambda,$$
where $H_\alpha(V) =(1-\alpha)H_\phi(V)+\alpha\trace(V^TD)H_\psi(V)+\alpha\psi(V)(DV^T+VD^T)$ and 
\begin{equation}
H_\phi(V)=\frac{2}{\trace(V^TBV)}(A-\phi(V)B),\quad H_\psi(V)=\frac{-\psi(V)}{\trace(V^TBV)}B \label{eq:sumTRnepv}
\end{equation}   
for the largest eigenpairs of $H_\alpha(V)$. From \eqref{eq:sumTRnepv}, it is immediate to compute the corresponding Fr\'echet derivatives using the chain rule to obtain 
\begin{eqnarray}
L_{H_\phi}(V,\dv) &=& -2\left(\frac{\trace(V^TB\dv)}{\trace(V^TBV)}H_\phi(V) + \frac{\trace(V^TH_\phi(V)\dv)}{\trace(V^TBV)}B\right),\label{eq:frechetHphi}\\
L_{H_\psi}(V,\dv) &=& -3\frac{\trace(V^TH_\psi(V)\dv)}{\trace(V^TBV)}B.\label{eq:frechetHpsi}
\end{eqnarray}
In \cite{uinepv}, the original NEPv using $H_\alpha(V)$ is replaced by an aligned function $G_\alpha(V)\define H_\alpha([\![V]\!])$ which is constructed to be orthogonally invariant. The alignment essentially consist of replacing $V$ by $V\mathcal{P}(V^TD)$, where $\mathcal{P}(Y)$ computes the orthogonal polar factor of $Y$. The Fréchet derivative of the aligned function $G_\alpha$ involves the Fr\'echet derivatives of $H_\phi$ and $H_\psi$ available from \eqref{eq:frechetHphi} and \eqref{eq:frechetHpsi}, respectively, as well as the Fr\'echet derivative of the alignment operator. The interested reader is referred to \cite[Cor. 5.1]{uinepv} for an explicit form of $L_{G_\alpha}(V,\dv)$. We also want to mention that source code for the performance of SCF on this type of problem is freely available from \url{https://github.com/ddinglu/uinepv}.\\
In our experiment, we copy the setup from Example 7.2, in which $n=3$, $k=2$ and
$$	A=\begin{bmatrix*}[r]-3.242 & -0.450 & 1.807\\-0.450 & -1.630 & 0.790\\1.807 & 0.790 & 0.226\end{bmatrix*},~
	B=\begin{bmatrix*}[r]0.592 &1.873 & 0.175\\1.873 & 6.332 & 0.617\\0.175 & 0.617 & 0.488\end{bmatrix*},~
	D=\begin{bmatrix*}[r]-1.430 & 2.768\\-0.120 & -0.630\\1.098 & 2.229\end{bmatrix*}.$$
In \autoref{fig:uinepv}, convergence results for $\alpha\in\lbrace 0.085,0.25,0.305,0.5,0.605,0.66\rbrace$ are displayed. Here, we only use JFNGK and found a tolerance of $10^{-2}$ or at most 20 steps of SCF to be an empirically food choice for switching.\\
One can see that neither SCF nor Newton method converges for $\alpha_4=0.5$ and convergence of SCF (solid line) is very slow for $\alpha_3=0.305$ as well as $\alpha_5=0.605$, whereas Newton's method (dashed line) manages to converge within only two to four steps after switching for $\alpha_1$, $\alpha_2$, $\alpha_5$ and $\alpha_6$ and requires eight steps for $\alpha_3$. Again, the advantage of the quadratic order of convergence over the linear order of convergence of SCF becomes obvious from the plots. However, we want to note that JFNGK was not able to beat SCF in terms of computation time for this problem due to the complexity of the evaluation of $L_{G_\alpha}(V)$, which, due to the alignment function, requires the computation of multiple singular value decompositions and solving Lyapunov equations, leading to extremely expensive solves since the Fr\'echet derivative has to be evaluated numerous times to expand the Krylov subspace.
\begin{figure}[h]
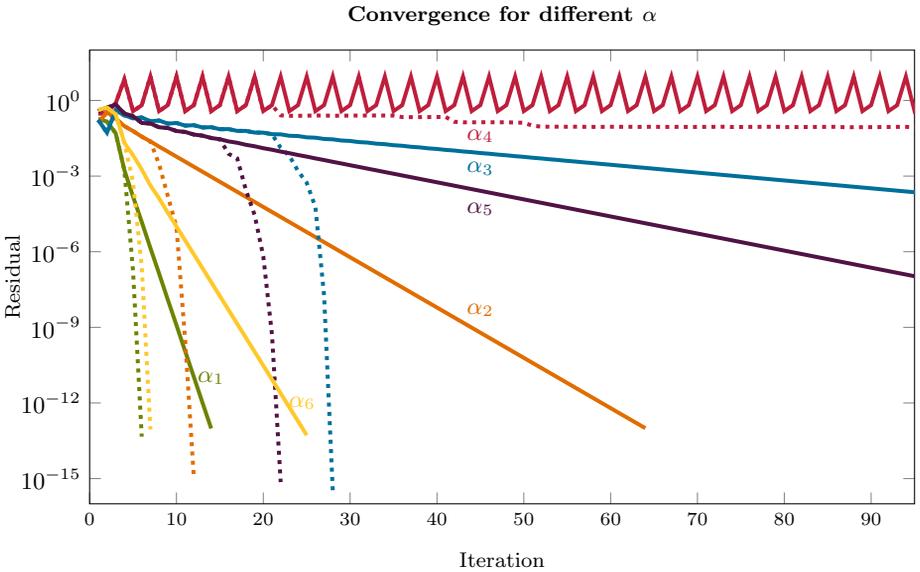

\centering
\include{figures/uinepv/residuals.tex}
\vspace{-1cm}%
\caption{Convergence of SCF(solid) and JFNGK(dashed) for different choices of $\alpha$ in the SumTR-problem}
\label{fig:uinepv}
\end{figure}
\subsection{A note on global LSQR as inner solver} \label{sec:lsqr}
To finish off this work, we want to give a brief comment on the use of different iterative solvers for the Newton correction equation \eqref{eq:nepvcorrection}. As noted in \autoref{lem:gmresls}, we require $\mathcal{R}(L_F(X_j))\,\cap\,\nullspace(L_F(X_j))=\lbrace 0\rbrace$ to be able to solve the update equation by Gl-GMRES without breakdown. In cases where this condition is not fulfilled, which is rarely the case very close to the solution in our experiments, it might be beneficial to use iterative solvers specifically designed to solve least squares problems, the most notable one being global LSQR \cite{lsqr,gllsqr}.
 While this solver does not require a full orthogonalization since it is based on the Golub-Kahan-bidiagonalization, it requires two function evaluations per Krylov-step: one with $L_F(X_j)$ and one with its adjoint $L_F^*(X_j)$, making it a lot more costly in cases where only small Krylov subspaces are required for convergence. Another problem is in finding the adjoint operator in the first place. For GNEPv, the adjoint can be explicitly formulated in terms of the adjoints of $L_G(V)$ and $L_H(V)$ by 
\begin{equation}
\resizebox{\textwidth}{!}{$L_F^*(X,\dx) = \begin{bmatrix}H(V)\dv + L_H^*(V)[\dv V^T] - \left(V(\dl+\dl^T)+G(V)\dv \Lambda^T + L_G^*(V)[\dv\Lambda V^T]\right)\\ -V^TG(V)\dv\end{bmatrix},$}\label{eq:adjointgnepv}
\end{equation}
which, for the case of NEPv ($G(V)=I_n,~L_G(V)\equiv 0$) reduces to 
\begin{equation}
L_F^*(X,\dx) = \begin{bmatrix}H(V)\dv + L_H^*(V)[\dv V^T] - \left(V(\dl+\dl^T)+\dv \Lambda^T \right)\\ -V^T\dv\end{bmatrix}. \label{eq:adjointnepv}
\end{equation}
The interested reader is referred to the appendix for a detailed description of GL-LSQR and the derivation of this formula.\\
In \autoref{fig:localsolve}, we have compared the relative residual for solving the update equation close to the optimal solution of a 3D-Kohn-Sham-problem of size $10^3\times8$ by global GMRES and global LSQR.
\begin{figure}[h]
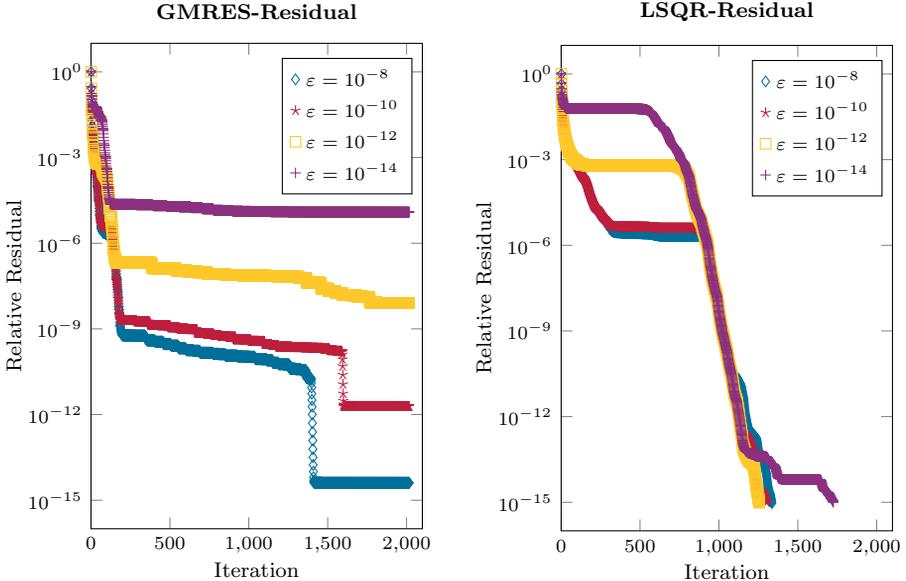

\begin{subfigure}[l]{0.495\textwidth}
\include{figures/compareks/1000x8GMRES.tex}
\end{subfigure}
\vspace{-.8cm}%
\begin{subfigure}[r]{0.495\textwidth}
\include{figures/compareks/1000x8LSQR.tex}
\end{subfigure}
\caption{Solution to $L_F(X,E)=-F(X)$ at $X=X^*+\varepsilon P$ by global GMRES (left) and global LSQR (right), 3D-Kohn-Sham-problem \eqref{eq:KS} of size $10^3\times 8$}
\label{fig:localsolve}
\end{figure}
For that, we choose $X=X^*+\varepsilon P$, where $X^*$ is the optimal solution computed by SCF, $\varepsilon$ is a small scaling parameter ranging from $10^{-8}$ to $10^{-14}$ and $P=\begin{bmatrix}P_1\\P_2\end{bmatrix}$ is a perturbation matrix with $P_1\in\RR^{n,k}$ being Gaussian and $P_2=\frac{1}{2}(\widehat{P}_2+\widehat{P}_2^T)$ being the symmetric part of the Gaussian matrix $\widehat{P}_2\in\RR^{k,k}$, and try to solve the update equation $L_F(X,E)=-F(X)$ at this point. In this experiment, we can see that GMRES does not find a solution with the desired tolerance $\tau=10^{-15}$ in all cases since the iteration stagnates before reaching the solution. We also observe that the stagnation happens earlier the smaller the perturbation is. For the smallest perturbation $\varepsilon=10^{-14}$, a relative accuracy of just $10^{-5}$ is reached, while for $\varepsilon=10^{-8}$, an accuracy of $10^{-14}$ can be achieved. On the other hand, LSQR always finds a solution with the desired accuracy of $10^{-15}$ after around 1300 steps for $\varepsilon\in\{10^{-8},10^{-10},10^{-12}\}$ and around 1700 steps for $\varepsilon=10^{-14}$ without full stagnation. Nonetheless, there is also a plateau after around 100-300 steps at a lower precision than the corresponding GMRES graph.\\
From these observations, we conclude that GMRES is the method of choice for our Newton method for multiple reasons: First of all, the case where GMRES stagnates rarely occurred in our experiments. Secondly, even if GMRES stagnates, we still observe improvement in the Newton iteration with the low accuracy solution. Since in practice we limit the maximum size of the Krylov subspace, taking multiple steps of GMRES at a low precision is usually faster than waiting for LSQR to converge to a high accuracy solution. Lastly, the inexact Newton framework frequently requires moderately accurate solutions, meaning that relative residuals of $10^{-4}$ or $10^{-5}$ suffice to get almost quadratic order of convergence close to the solution. This means that we do not have to solve really close to the solution since the quadratic order of convergence can give us accurate solutions at $X_j$ if $\lVert F(X_{j-1})\rVert_F\approx 10^{-9}$, which, assuming a moderate level of smoothness, compares to the curves for $\varepsilon=10^{-8}$ or $\varepsilon=10^{-10}$, and the solution at the start of the plateau usually satisfies the moderate forcing accuracy, allowing for smaller subspaces and reliable solves.\\
However, if one observes convergence issues with the proposed method, using LSQR instead should salvage this issue and lead to better results if we choose the accuracy high enough. Using GMRES first and switching to LSQR close to the solution is also an option, if one has a priori knowledge of local convergence issues and sophisticated switching criteria.       	  
\section{Conclusion}
In this paper, we have introduced an inexact Matrix-Newton method for solving NEPv and GNEPv. For this algorithm, we have reported on how to implement the algorithm efficiently by using a matrix-free global GMRES approach for solving the update equation and given a result indicating local quadratic convergence of the algorithm. The algorithm was compared to the SCF algorithm in a variety of numerical experiments using MATLAB. Here we could see that the Newton algorithm is able to compete with SCF in most applications where SCF has linear convergence speed and yield optimal solutions, even in some cases where SCF failed to converge. While these experiments indicate the advantages of the Newton approach, we also mentioned that a reliable rule for switching between SCF and Newton's method is still a point of current research that is crucial to the performance of the algorithm in many applications. The work was finished off by a few comments on alternative solvers, such as LSQR, that can give higher accuracy solutions but proved to be a lot slower in practice. 
\section*{Acknowledgments}
The author wants to thank Dr. Philip Saltenberger for the numerous conversations and input that helped initiate this work. The author also wants to thank Prof. Ding Lu for the insightful discussions and his comments on earlier versions of this manuscript as well as Lemma 3. A special thanks goes to the anonymous referees whose reports significantly improved the quality of this work.
\section*{Code Availability}
Source code supporting this paper is available from \url{https://doi.org/10.5281/zenodo.10124821} and has been updated for the revised version of this manuscript.
\appendix
\section{Proof of \autoref{lem:consistency}}
\begin{enumerate}
\item[(i)] The first part is immediately clear from the discussion above \autoref{lem:consistency} since 
$$L_{F_2}(V,\dv)=-(V^T\dv + \dv^T V) = -(W+W^T)\equiv -2S$$ is always symmetric, where $W=V^T\dv$ and $S=\frac{1}{2}(W+W^T)$ is the symmetric part of $W$.
\item[(ii)] Let $B\define \begin{bmatrix} Y\\-2S\end{bmatrix}$, where $Z\in\RR^{n,k}$ and $S\in\mathbb{S}_k$. We will show that the equation $L_F(X,\dx)=B$ always has a solution when $X$ is sufficiently close to $X^*$.\\
To start off, let $\dx_1\define\begin{bmatrix}V S\\0\end{bmatrix}$. Then we have from \eqref{eq:nepvcorrection} 
\begin{equation}
	L_F(X,\dx_1) = \begin{bmatrix} H(V)VS + L_H(V,VS)V - VS\Lambda \\ -(V^TVS + S^TV^TV)\end{bmatrix} \equiv\begin{bmatrix}Y_1 \\-2S\end{bmatrix}.	 
\end{equation}
Now consider the modified operator from \eqref{eq:nepvcorrectionmod}. This operator is obtained by replacing the orthogonality constraint $V^T V=I$ with the normalization constraint $C^TV=I$ w.r.t. an arbitrary matrix $C\in\RR^{n,k}$ as suggested in \cite{elias_implicit}. In our particular case, we choose $C=V$.\\  Since $L_{\widetilde{F}}(X^*)$ is locally invertible, the linear system 
\begin{equation}
	L_{\widetilde{F}}(X,\dx_2)=\begin{bmatrix}Y-Y_1\\0\end{bmatrix} \label{eq:proof1}
\end{equation}
always has a unique solution $\dx_2=(\dv_2,\dl_2)$. By construction of $L_{\widetilde{F}}(X)$, we need to have $V^T\dv_2=0$ to satisfy the bottom part of \eqref{eq:proof1}. In consequence,
$$ L_F(X,\dx_2)=L_{\widetilde{F}}(X,\dx_2)=\begin{bmatrix}Y-Y_1\\0\end{bmatrix}.$$
By linearity of the Fr\'echet derivative, we finally get 
\begin{equation}
 	L_F(X,\dx_1+\dx_2) = L_F(X,\dx_1) + L_F(X,\dx_2) = \begin{bmatrix}Y_1\\-2S	\end{bmatrix} + \begin{bmatrix}Y - Y_1\\0\end{bmatrix} =B
\end{equation}
and $\dx\define \dx_1+\dx_2$ is a solution of \eqref{eq:nepvcorrection}.
\end{enumerate}
\section{Global LSQR and derivation of the adjoint Fréchet derivative} 
\begin{algorithm}[H]
		\SetAlgoVlined
	\KwIn{$B\in\RR^{m,p}$, $\mathcal{A}:\RR^{n,k}\rightarrow\RR^{m,p}$, $\mathcal{A}^*:\RR^{m,p}\rightarrow\RR^{n,k}$}
	\KwOut{$X^*\in\RR^{n,k}$ approximate least-squares solution to $\min_{X\in\RR^{n,k}}\lVert \mathcal{A}(X)-B\rVert_F$}
	Set $X_0=0_{n,k}$\\
	Compute $\beta_1=\lVert B\rVert_F$ and normalize $U_1=\frac{B}{\beta_1}$\\
	Compute $V=\mathcal{A}^*(U_1)$ and $\alpha_1=\lVert V\rVert_F$, normalize $V_1=\frac{V}{\alpha_1}$\\
	Set $W_1=V_1$,\quad $\overline{\phi}_1=\beta_1$,\quad $\overline{\rho}_1=\alpha_1$\\
	\For{$i=1,2,\dots,\ell$}{
		Compute $U=\mathcal{A}(V_i)-\alpha_i U_i$ and $\beta_{i+1}=\lVert U\rVert_F$, normalize $U_{i+1}=\frac{U}{\beta_{i+1}}$\\
		Compute $V=\mathcal{A}^*(U_{i+1})-\beta_{i+1}V_i$ and $\alpha_{i+1}=\lVert V\rvert_F$, normalize $V_{i+1}=\frac{V}{\alpha_{i+1}}$\\
		Compute Givens rotations $\rho_i=\sqrt{\overline{\rho}_i^2 + \beta_{i+1}^2}$,\quad $c_i=\frac{\overline{\rho}_i}{\rho_i}$,\quad $s_i=\frac{\beta_{i+1}}{\rho_i}$\\
		$\theta_{i+1}=s_i\alpha_{i+1}$,\quad $\overline{\rho}_{i+1}=c_i\alpha_{i+1}$,\quad $\phi_i = c_i\overline{\phi}_i$,\quad $\overline{\phi}_{i+1}=-s_i\overline{\phi}_i$\\
		Update $X_{i}=X_{i-1} + \frac{\phi_i}{\rho_i} W_i$ and $W_{i+1}=V_{i+1} - \frac{\theta_{i+1}}{\rho_i}W_i$\\		
		\If{$\lvert \overline{\phi}_{i+1}\rvert<\text{tol}$}{
			\Return $X^*=X_i$
		}
	}
\vspace{-.2cm} 
	\caption{Global LSQR for the linear operator equation $\mathcal{A}(X)=B$}
	\label[algorithm]{alg:gllsqr}
\end{algorithm}
As mentioned earlier, using \autoref{alg:gllsqr} to solve \eqref{eq:nepvcorrection} requires the evaluation of the adjoint $L_F^*(X)$ of the Fréchet derivative of $F(X)$ to extend the Golub-Kahan-bidiagonalization. In this section, we will derive the formula for $L_F^*(X)$ in the case of GNEPv stated in \eqref{eq:adjointgnepv}, the case of NEPv \eqref{eq:adjointnepv} then directly follows by setting $G(V)=I_n$.\\
For GNEPv, we have 
\begin{eqnarray*}
L_F(X,\dx) &=& \begin{bmatrix}H(V)\dv + L_H(V)[\dv]V - G(V)\left(V\dl + \dv\Lambda\right) - L_G(V)[\dv]V\Lambda\\-(V^T\dv+\dv^TV \end{bmatrix}\\
		   &=& \underbrace{\begin{bmatrix} H(V) & -G(V)V \\ -V^T & 0  \end{bmatrix}}_{=:A(X)}\dx + \underbrace{\begin{bmatrix}L_H(V)[\dv] - G(V)\dv\Lambda - L_G(V)[\dv]V\Lambda\\-\dv^TV \end{bmatrix}}_{=:C_X(\dx)}
\end{eqnarray*}
Defining $S_X(\dx):=A(X)\dx$, it is immediate to see that $L_F^*(X,\dx) = S_X^*(\dx) + C_X^*(\dx)$ and, since $A(X)$ is not depending on $\dx$, $S_X^*(\dx)=A(X)^T\dx$. Thus, we just have to find the adjoint $C_X^*$ of $C_X$, which we will now do: Let $\dx_1,\dx_2\in\RR^{n+k,k}$. Then, we have 
\begin{eqnarray*}
\langle \dx_1,C_X(\dx_2)\rangle_F &=& \trace\left(\dx_1^T \begin{bmatrix}L_H(V)[\dv_2]V - G(V)\dv_2\Lambda - L_G(V)[\dv_2]V\Lambda\\-\dv_2^TV \end{bmatrix}\right)\\
							&=& \trace\left(\dv_1^TL_H(V)[\dv_2]V - \dl_1^T\dv_2^TV\right)\\
							&{}& -\trace\left(\dv_1^TG(V)\dv_2\Lambda + \dv_1^TL_G(V)[\dv_2]V\Lambda\right)\\
							&=&\langle\dv_1V^T, L_H(V,\dv_2)\rangle_F - \langle \dv_1\Lambda^TV^T,L_G(V)[\dv_2]\rangle_F\\
							&{}& -\langle V\dl_1^T+G(V)\dv_1\Lambda^T,\dv_2\rangle_F \\
							&=& \resizebox{.72\textwidth}{!}{$\left\langle \begin{bmatrix}L_H^*(V)[\dv_1V^T] -\left(V\dl_1^T+G(V)\dv_1\Lambda^T+L_G^*(V)[\dv_1\Lambda V^T]\right)\\0\end{bmatrix},\dx_2\right\rangle_F$}\\
							&\equiv& \langle C_X^*(\dx_1),\dx_2\rangle_F.							
\end{eqnarray*} 
Putting together $C_X^*(\dx)$ and $S_X^*(\dx)$, we get the adjoint
\begin{equation*}
\resizebox{\textwidth}{!}{$L_F^*(X,\dx)=\begin{bmatrix}H(V)\dv +L_H^*(V)[\dv V^T]- \left(V(\dl+\dl^T)+G(V)\dv\Lambda^T+L_G^*(V)[\dv\Lambda V^T]\right)\\-V^TG(V)\dv\end{bmatrix}$}.
\end{equation*}
\bibliographystyle{plain}

\end{document}